\documentclass[10pt]{amsart}
\usepackage{graphicx}
\usepackage{amssymb,amsmath,amsxtra,latexsym, MnSymbol,psfrag}
\usepackage{xcolor}
\usepackage{cite}
\usepackage[curve]{xypic}
\def\dint{\displaystyle\int}
\def\dsum{\displaystyle\sum}
\def\dprod{\displaystyle\prod}

\def\ub{\textbf u}

\newtheorem{thm}{Theorem}%[section]
\newtheorem{thms}{Theorem}[section]
\newtheorem{lemma}{Lemma}[section]
\newtheorem{Def}{Definition}[section]
\newtheorem{prop}{Proposition}[section]
\newtheorem{cor}{Corrolary}[section]
\newtheorem{remark}{Remark}[section]

\newcommand{\ualpha}{\underline{\alpha}}
\newcommand{\ubeta}{\underline{\beta}}

\newcommand{\ev}{\mathrm{ev}}

\newcommand{\KKeff}{\mathbb{K}^{\text{\rm eff}}}

\newcommand{\Mbar}{\overline{\mathcal{M}}}

\DeclareMathOperator{\Res}{Res}

\newcommand{\vir}{\text{\rm vir}}

\newcommand{\CR}{{\operatorname{CR}}}

\newcommand{\PP}{\mathbb{P}}
\newcommand{\QQ}{\mathbb{Q}}
\newcommand{\bT}{\mathbb{T}}
\newcommand{\RR}{\mathbb{R}}

\newcommand{\KK}{\mathbb{K}}
\newcommand{\TT}{\mathbb{T}}
\newcommand{\LL}{\mathbb{L}}
\newcommand{\ZZ}{\mathbb{Z}}

\newcommand{\CC}{\mathbb{C}}
\newcommand{\EE}{\mathbb{E}}
\newcommand{\tx}{\tilde{x}}

\newcommand{\tX}{\widetilde{X}}

\renewcommand{\ev}{\operatorname{ev}}

\newcommand{\rank}{\operatorname{rank}}

\newcommand{\age}{\operatorname{age}}

\newcommand{\cA}{\mathcal{A}}

\newcommand{\calD}{\mathcal{D}} 

\newcommand{\cF}{\mathcal{F}} 
\newcommand{\cI}{\mathcal{I}}
\newcommand{\cQ}{\mathcal{Q}}

\newcommand{\cS}{\mathcal{S}}

\newcommand{\cX}{\mathcal{X}}

\renewcommand{\cL}{\mathcal{L}}

\newcommand{\bt}{\mathbf{t}}
\newcommand{\bN}{\mathbf{N}}

\newcommand{\bE}{\mathbf{E}}

\newcommand{\w}{{\mathbf{w}}}

\newcommand{\bu}{\mathbf{u}}
\newcommand{\fh}{\mathfrak{h}}
\begin{document}

\title{Equivariant mirror symmetry for footballs}
\date{\today}
\author{Zhuoming Lan}
\address{Zhuoming Lan, Department of Mathematical Sciences, Tsinghua University, Haidian District, Beijing 100084, China}
\email{lanzm21@mails.tsinghua.edu.cn}
\maketitle
\begin{abstract}

In this paper, we establish equivariant mirror symmetry for footballs $\cF(m,r)$. This extends the results by B. Fang, C.C. Liu and Z. Zong, where the projective line was considered [{\it Geometry \& Topology} 24:2049-2092, 2017], and the results by D. Tang of weighted projective lines, on [arXiv:1712.04836].  More precisely, we prove the equivalence of the $R$-matrices for A-model and B-model at large radius limit, and establish isomorphism for $R$-matrices for general radius. We further demonstrate that the graph sum of higher genus cases are the same for both models, hence establish equivariant mirror symmetry for footballs. In last two sections the large radius limit and equivariant limit are considered, resulting a generealized Bouchard-Mari\~{n}o conjecture and Norbury-Scott conjecture respectively.  
\end{abstract}

\tableofcontents

\section{Introduction}

In mirror symmetry for toric varieties, one aims at establishing equivalence between A-model invariants and the LG B-model invariants for a given toric variety.

On the A-model side, there are extensive studies for the equivariant Gromov-Witten theory of toric varieties. In \cite{Given1}, A.B. Givental computed all genus descendent invariants of equivariant Gromov-Witten theory of tori action with isolated fixed points for the smooth case. The process of recovering higher genus data is now known as Givental's formula. In \cite{CCIT15}, Coates, Corti, Iritani, and Tseng estabished the mirror theorem for all toric stacks,  which enables the calculation of $J$-function via the mirror map. In \cite{Zong}, Z. Zong gave all genus equivariant Gromov-Witten invariants for GKM orbifolds by generalizing Givental's formula.

On the B-model side, B. Eynard and N. Orantin discovered a way to compute the topological expansion of matrix integrals in \cite{Eynard}. Eynard-Orantin's topological recursion is related to Givental's formula \cite{DOSS}.

B. Fang, C.C. Liu and Z. Zong established the equivariant mirror symmetry for the projective line. They directly computed the $R$-matrices of both A and B-models, and applied Givental's formula and Eynard-Orantin's recursion, thereby proved the equivariant mirror symmetry for the projective line by calculating graph sums  \cite{FLZ17}.

In this paper, we extend the equivariant mirror symmetry to the weighted projective line.
First, we associate the equivariant Gromov-Witten invariants of the weighted projective line to the Eynard-Orantin invariants of the affine curve determined by the superpotential of its $T$-equivariant Landau-Ginzburg mirror. It is proved by calculating the graph sum and applying the main results in \cite{Eynard}\cite{Given}\cite{Given1}. We use the equivalence of $R$-matrices in both models, which is established by computations with quantum Riemann-Roch, and integration on Lefschetz thimble at the large radius limit \cite{Tseng10}.

Secondly, we establish a precise correspondence between A-model genus $g, n$ point descendent equivariant Gromov-Witten invariants, and the Laplace transform of B-model Eynard-Orantin invariant along Lefschetz thimbles.  This extends the result in \cite{Fang20}.

Thirdly, we consider the large radius limit of the mirror symmetry. On the A-model side $\cF(m,r)$ reduces to $[\CC/\mu_{r}]$, and the $F_{g,n}$ corresponds to the orbifold Hurwitz numbers defined in \cite{JPT11}. This established a correspondence between $r$-Hurwitz numbers with B-model topological recursion. When further letting $q_{i}\to 0$, the $r$-Bouchard-Mari{\~n}o conjecture in \cite{BSLM14} is recovered.

Fourthly, we consider the non-equivariant limit of the mirror symmetry, with $q_{i}\to0$. Similar to the method in \cite{FLZ17}, we can explicitly calculate the expansion of $\omega_{g,n}$, resulting a generalized version of Norbury-Scott conjecture.

\color{black}
\subsection*{Acknowledgments}
 The author would like to thank Dun Tang for the fundamental work on weighted projective lines, which is very much similar to footballs. The author want to express his great gratitude to Professor Chiu-Chu Mellisa Liu, who provided a lot of helpful idea in the whole process. The  author would also like to thank Professor Bohan Fang and Professor Zhengyu Zong for helpful discussions. The author is financially supported by the Tsinghua Scholarship for Overseas Graduate Studies. 

\section{The Frobenius Manifold}
% We use bold form quantities $\bold{w_\ell,w_{-\ell},p,q_\ell,q_0,q_{-\ell}}$ to represent cohomology classes and normal form quantities $w_1,w_2,p,q_l,q_0,q_{-l}$  to represent complex numbers.

\subsection{Geometry of $\mathcal{F}(m,r)$ as a toric orbifold}\label{sec:fan}

A football $\mathcal{F}(m,r)$ is a toric orbifold defined by the following stacky fan \cite{BCS05}: 
% given by the fan below by standard toric construction:
\begin{figure}[h]
\begin{center}
\setlength{\unitlength}{2mm}
\begin{picture}(20,5)
%\linethickness{1pt}
\put(12,4){\vector(-1,0){10}}
\put(12,4){\vector(1,0){6}}
\put(12,3){${}^|$}
\put(1,1){$-m$}
\put(17,1){$r$}
\end{picture}
   %\caption{}
  \label{fig0}
\end{center}
\end{figure}

i.e., $\mathcal{F}(m,r)=(\mathbb{C}^2\setminus\{(0,0)\})/G_{m,r}$, where $G_{m,r}$ is defined by $$G_{m,r}=\{(t_1,t_2)|(t_1,t_2)\in(\mathbb{C}^{*})^2,t_1^{r} t_2^{-m}=1\}.$$ From this construction we have a $T=(\mathbb{C}^\ast)^2$ action on $\mathcal{F}(m,r)$ by $(t_1,t_2)\cdot[z_1,z_2]=[t_1z_1,t_2z_2]$.

The football $\cF(m,r)$ can also be described
in terms of an extended stacky fan \cite{Jiang08}.

\begin{Def}
 An \emph{$S$-extended stacky fan} 
is a quadruple $\mathbf{\Sigma}= (\bN,\Sigma,\beta,S)$, where:
\begin{itemize} 
\item $\bN$ is a finitely generated abelian group\footnote{Note that $\bN$ may have torsion.}; 
\item $\Sigma$ is a rational simplicial fan in $\bN\otimes \RR$;   
\item $\beta\colon \ZZ^M \to \bN$ is a homomorphism; 
we write $b_i = \beta(\tilde{b}_i)\in \bN$ for the image of the $i$th standard 
basis vector $\tilde{b}_i\in\ZZ^M$,
and write $\overline{b}_i$ for the image of $b_i$ in $\bN\otimes \RR$; 
\item $S \subset I$ is a subset, where $I$ is an index set with cardinality $M$.
\end{itemize} 
such that:
\begin{itemize} 
\item each one-dimensional cone of $\Sigma$
is spanned by $\overline{b}_i$ for a unique 
$i\in I\setminus S$, and
each $\overline{b}_i$ with 
$i\in I\setminus S$ spans a one-dimensional 
cone of $\Sigma$; 

\item for $i\in S$, $\overline{b}_i$ lies in the support $|\Sigma|$ 
of the fan. 
\end{itemize}    
\end{Def}

$\mathcal{F}(m,r)$ can be also constructed through the following stacky fan from the following data:
\begin{itemize}
  \item an $(m+r)$-dimensional algebraic torus $K\cong (\mathbb{C^\ast})^{m+r}$. Set $\LL=\mathrm{Hom}(\mathbb{C}^\ast,K)$;
  \item $M=m+r+1$ elements $D_{-m}, D_{-m+1},\ldots, D_r\in \mathbb{L}^{\vee} =\mathrm{Hom}(K,\mathbb{C}^\ast)$, such that $\mathbb{L}^{\vee}\otimes\mathbb{R}=\dsum\mathbb{R}\cdot D_i$;
  \item a vector $\omega\in \mathbb{L}^{\vee}\otimes\mathbb{R}$ (that defines the stability condition).
\end{itemize}
In $\cF(m,r)$ case we first consider the exact sequence:
\begin{equation}\label{eq:exact}
  \xymatrix{
    0 \ar[r] &
    \LL \ar[r]^{i} &
    \ZZ^{m+r+1} \ar[r]^\beta & 
    \bN \ar[r] & 
    0 }
\end{equation}

and the corresponding exact sequence of algebraic tori:
\begin{equation}\label{eq:exact-tori}
  \xymatrix{
    0 \ar[r] &
    K \ar[r] &
   \mathbb{T} \ar[r] & 
    \cQ \ar[r] & 
    0 }
\end{equation}
where  $\LL\simeq\ZZ^{m+r}$ and $ \bN=\ZZ$. 
Writing the $\ZZ$-basis of  $\LL^{\vee}$ to be 
$\{e_{-m},e_{-m+1},\dots,e_{-2},e_0,e_1,\dots,e_{r}\}$. 
The map $i=(D_{-m},D_{-m+1},\ldots,D_r)$ with $D_i=e_i$ for $i\in 
[-m,r]\setminus\{-1\}$ and $D_{-1}=\sum_{i\in [-m,r]\setminus\{-1\}}ie_{i}$ where 
$[-m,r]:=\{ -m, -m+1,\ldots, r\}$, and $\beta=(b_{-m},b_{-m+1},\dots,b_{r})$ with $b_i=i\in\bN=\ZZ$.

And $K=\LL\otimes\CC^{*}\simeq(\CC^{*})^{m+r}$, $\TT=(\CC^{*})^{m+r+1}$, $\cQ:=\TT/K\simeq\CC^{*}$

The extended stacky fan
$\mathbf{\Sigma}_\omega=(\bN, \Sigma_\omega, \beta,S)$
corresponding to $\mathcal{F}(m,r)$ 
consists of the group $\bN$ and the map $\beta$ defined 
above, together with a fan $\Sigma_\omega$ in $\bN\otimes\RR$ 
and $S$ given by\footnote{This is why we refer to the 
elements of $\cA_\omega$ as anticones.}: 
\begin{align*} 
\Sigma_\omega & = \{\sigma_{I} : \overline{I} \in \cA_\omega\}, \\ 
S & = \{ i \in \{1,\dots,m\} : \overline{\{i\}}  
\notin \cA_\omega\}. 
\end{align*} 
The definition of $\cA_\omega$ follows  \cite[Definition 4.2]{CIJ}. In our case, $\cA_\omega=\{\{-m+1,-m+2,\dots,r\},\{-m,-m+1,\dots,r-1\},\{-m,-m+1,\dots,r\}\}$ the subset 
$S= \{-m+1,-m+2,\dots,r-1\}$ .

With these data we have $\cX_{\omega}=\cF(m,r).$

\subsection{Equivariant cohomology 
of $\mathcal{F}(m,r)$}
%In this section we follow the conclusion of Section 4.3 of \cite{CIJ} on $H^{*}_{\cQ}(X_{\omega};\CC)$ and $H^{*}_{\TT}(X_{\omega};\CC)$. 

Let $\bullet$ denote a point. The surjective group homomorphism $\TT\to \cQ$ induces an injective ring homomorphism
\begin{eqnarray*}
H^*_{\cQ}(\bullet;\CC) = \CC[p] &\longrightarrow&  
H^*_{\TT}(\bullet;\CC) =\mathbb{C}[\w_{-m},\cdots,\w_{-1},\w_{0},\w_1,\cdots,\w_{r}] =: \CC[\w]\\
p &\longmapsto & \mathbf{p} 
= \sum_{\ell=1}^{r} \ell \bold{w}_\ell - \sum_{\ell=1}^m \ell \bold{w}_{-\ell} =
\bold{p}_+ -\bold{p}_-
\end{eqnarray*}
where $\mathbf{p}_{+} =\sum_{\ell=1}^{r} \ell \bold{w}_\ell$ and 
$\bold{p}_{-}=\sum_{\ell=1}^m \ell \bold{w}_{-\ell}$.

The $\cQ$-equivariant and $\TT$-equivariant cohomology rings
of $\cF(m,r)$ are given by (see e.g. \cite[Section 4.3]{CIJ}): 
\begin{align*}
    H^{*}_{\cQ}(\cF(m,r),\CC)&=\CC[u_1,u_2, p]/\langle u_1 u_2, p-(ru_1-mu_2) \rangle ,\\
    H^{*}_{\TT}(\cF(m,r),\CC)&=
    H^*_{\cQ}(\cF(m,r);\CC)
    \otimes_{\CC[p]} \CC[\w] =
    \CC[u_1, u_2, \mathbf{w}]/\langle u_1 u_2, \bold{p}-(ru_1-mu_2) \rangle.
\end{align*}

The $\cQ$-equivariant and $\TT$-equivariant
Chen-Ruan orbifold
cohomology rings of 
$\cF(m,r)$ are given by (see e.g. \cite[Section 8.8]{Liu13}):
\begin{align*}
    H^{*}_{\CR,\cQ}(\cF(m,r);\CC)&=\CC[y_1,y_2, p]/\langle y_1 y_2, p-(ry_1^{r}-mu_2^{m}) \rangle, \\
    H^{*}_{\CR,\TT}(\cF(m,r),\CC)&=
    H^*_{\CR,\cQ}(\cF(m,r);\CC)\otimes_{\CC[p]}\CC[\w]=
    \CC[y_1, y_2, \mathbf{w}]/\langle y_1 y_2, \bold{p}-(ry_1^{r}-my_2^{m}) \rangle.
\end{align*}

% For more details, please refer to \cite{Adem07} and the original papers \cite{Chen04,Chen01,Chen}.

\subsection{The $\TT$-equivariant superpotential and its Jacobian ring}

For $y\in\mathbb{C}$, let $Y=e^y\in \mathbb{C}^\ast$.
Define the $\TT$-equivariant superpotential $W_{\TT}:\mathbb{C}^\ast\to \mathbb{C}$ by
\[
W_{\mathbb{T}}(Y)= Y^{r}+ \dsum_{\ell=1}^{r-1}\tilde{q}_\ell Y^\ell + \tilde{q}_0+\dsum_{\ell=1}^{m}\tilde{q}_{-\ell} Y^{-\ell} - \tilde{w}_{r}\log (Y^{r}) -\dsum_{\ell=1}^{r-1}\tilde{w}_\ell\log(\tilde{q}_\ell Y^\ell) - \dsum_{\ell=0}^{m} \tilde{w}_{-\ell}\log(\tilde{q}_{-\ell}  Y^{-\ell}).
\]

Let $x=W_\TT(e^y)$. The Jacobian ring of $W_{\mathbb{T}}$ is
\[
\mathrm{Jac}(W_\mathbb{T})= \mathbb{C}[w][Y,Y^{-1}]/\left<\dfrac{\partial W_\mathbb{T}}{\partial y}\right> = \mathbb{C}[w][Y, Y^{-1}]/\left<rY^{r} +  \dsum_{\ell=1}^{r-1}\ell \tilde{q}_\ell Y^\ell - \dsum_{\ell=1}^{m} \ell \tilde{q}_{-\ell} Y^{-\ell} -\tilde{p}\right>,
\]
where $\tilde{p}=\sum_{\ell=1}^{r} \ell \tilde{w}_\ell - \sum_{\ell=1}^m \ell \tilde{w}_{-\ell}$.

Let $\{P_\alpha\}$ be the set of critical points of $W_{\mathbb{T}}$. Define residual pairing $(f,g)$ on $\mathrm{Jac}(W_\TT)$ by 
\[
(f,g)= \dsum_{\alpha} \mathrm{Res}_{P_\alpha} \dfrac{f(Y)g(Y)}{\big({\partial W_\TT}/{\partial y}\big)} \dfrac{dY}Y.
\]

%Let $W$ denote the non-equivariant limit of $W_{\bT}$, i.e., the limit of $w_\ell \rightarrow 0$.

\subsection{Equivariant Gromov-Witten invariants}
Given $g, N, d\in \ZZ_{>0}$, let $\Mbar_{g,N}(\cF(m,r),d)$
be the moduli stack of genus $g$, $N$-pointed, degree $d$ twisted stable maps to $\cF(m,r)$ \cite{AGV02, AGV08}. (Note that $\Mbar_{g,N}(\cF(m,r),d)$ is empty if  $d=0$ and $2g-2+N<0$.) Let $\mathrm{ev}_j: \Mbar_{g,n}(\cF(m,r),d) \to \cI \cF(m,r)$ be the evaluation at the $j$-th
marked point, where $\cI\cF(m,r)$ is the inertia stack of $\cF(m,r)$.
The $\TT$-action on $\cF(m,r)$ induces
$\TT$-actions on $\Mbar_{g,N}(\cF(m,r),d)$ and on $\cI\cF(m,r)$, and
the evaluation maps $\ev_j$ are $\TT$-equivariant, so we have pullback maps
$$
\ev^*_j: H^*_{\CR,\TT}(\cF(m,r);\CC)
\cong H^*_{\TT}(\cI \cF(m,r);\CC)
\longrightarrow H^*(\Mbar_{g,N}(\cF(m,r),d);\CC)
$$
which are morphism of $\CC[\w]$-modules. 
Here $\cong$ is an isomorphism of 
$\CC[\w]$-modules but not a ring isomorphism. 

Let $\Mbar_{g,N}(\PP^1,d)$ be the moduli stack of genus $g$, $N$-pointed,
degree $d$ stable maps to the projective line $\PP^1$, the coarse moduli of the football $\cF(m,r)$ for any $m,r\in \ZZ_{>0}$. Let $\pi:\Mbar_{g,N+1}(\PP^1,d) \to \Mbar_{g,N}(\PP^1,d)$ be the universal curve, and let 
$\omega_\pi$ be the relative dualizing sheaf. For $j=1,\ldots, N$, let
$s_j:\Mbar_{g,N}(\PP^1,d)\to \Mbar_{g,N+1}(\PP^1,d)$ be the section associated to the $j$-th marked point. Then $\TT$ acts on $\PP^1$ and $\Mbar_{g,N}(\PP^1,d)$, and $\pi$ and $s_j$ are $\TT$-equivariant. 
For $j=1,\ldots, N$, let
$$
\LL_j = s_j^*\omega_\pi
$$
be the $j$-th tautological bundle, whose fiber over a moduli point $[f: (\Sigma,x_1,\cdots,x_N)\to \PP^1]$ in $\Mbar_{g,N}(\PP^1,d)$ is the cotangent line $T_{x_j}^\ast \Sigma$ of the domain $\Sigma$  at the $j$-th marked point $x_j$, and let
$$
\psi^{\TT}_j :=(c_1)_\TT(\LL_j) \in H^2_{\TT}(\Mbar_{g,N}(\PP^1,d);\QQ)
\subset H^2_\TT(\Mbar_{g,N}(\PP^1,d);\CC). 
$$ 
There is a $\TT$-equivariant map $p:\Mbar_{g,n}(\cF(m,r),d)\to \Mbar_{g,N}(\PP^1,d)$. Following \cite[Section 2.5.1]{Tseng10}, we define
$\TT$-equivariant descendant classes
$$
\hat{\psi}_j^{\TT}:= p^*\psi_i^{\TT} \in H^2_{\TT}(\Mbar_{g,N}(\cF(m,r),d);\QQ)
\subset H^2_{\TT}(\Mbar_{g,N}(\cF(m,r),d);\CC), \quad j= 1,\ldots,N. 
$$
The non-equivariant descendant class $\hat{\psi}_j \in 
H^2(\Mbar_{g,n}(\cF(m,r),d);\QQ)$ is denoted
$\bar{\psi}_j$ in \cite{Tseng10, CCIT15} and $\psi_j$ in \cite{Zong}.

Given  $\gamma_1,\cdots,\gamma_N \in
H_{\CR,\TT}^\ast(\mathcal{F}(m,r);\CC)$  and $a_1,\cdots, a_n\in \ZZ_{\geq 0}$, define the $\TT$-equivariant orbifold descendent Gromov-Witten invariants as 
\[
\left< \tau_{a_1}(\gamma_1)\cdots \tau_{a_n}(\gamma_N) \right>_{g,N,d}^{\mathcal{F}(m,r),\TT} = \dint_{[\Mbar_{g,N}(\cF(m,r),d)]^w} \dprod_{j=1}^N (\hat{\psi}^\TT)_j^{a_j} \mathrm{ev}_j^\ast(\gamma_j) \in \CC[\bold{w}].
\]
where $[\Mbar_{g,N}(\cF(m,r),d]^w$ is the weighted virtual fundamental class (see \cite[Section 4.6]{AGV02} and
\cite[Section 2.5.1]{Tseng10}).

\subsection{Generating functions}
For $2g-2+ M + N>0$ and given $\gamma_1,\cdots,\gamma_{N+M}\in H_{\CR,\TT}^\ast(\cF(m,r))$, define
\[
\begin{split}
& \left<  \dfrac{\gamma_1}{z_1-\hat{\psi}_1},\cdots,\dfrac{\gamma_N}{z_N-\hat{\psi}_N},\gamma_{N+1},\cdots,\gamma_{N+M} \right>_{g,M+N,d}^{\cF(m,r),\TT} \\
=& \dsum_{a_1,\cdots,a_N\geq 0} \left< \tau_{a_1}(\gamma_1)\cdots\tau_{a_N}(\gamma_N)\tau_0(\gamma_{N+1})\cdots\tau_0(\gamma_{N+M}) \right>_{g,M+N,d}^{\mathcal{F}(m,r),\TT} \dprod_{j=1}^N z_j^{-a_j-1}.
\end{split}
\]
This is actually the formal expansion of $\frac{\gamma_i}{z_i-\hat{\psi}_i}$ at $z_i^{-1}=0$.
In the unstable case $g=0,d=0,M+N=1 \;\mathrm{ or } \;2$, we define
\[
\begin{split}
\left< \frac{\gamma_1}{z_1-\hat{\psi}_1}\right>_{0,1,0}^{\mathcal{F}(m,r),\TT} & = z_1\int_{\mathcal{F}(m,r)} \gamma_1,\\
\left< \frac{\gamma_1}{z_1-\hat{\psi}_1},\gamma_2\right>_{0,2,0}^{\cF(m,r),\TT} & = \int_{\cF(m,r)} \gamma_1\cup\gamma_2,\\
\left< \frac{\gamma_1}{z_1- \hat{\psi}_1}, \dfrac{\gamma_2}{z_2-\hat{\psi}_2}\right>_{0,2,0}^{\cF(m,r),\TT} & = \dfrac1{z_1+z_2}\int_{\cF(m,r)} \gamma_1\cup\gamma_2.
\end{split}
\]
%This is motivated by the special case of $g=0,d=0,M+N\geq 3$:
%\[
%\begin{split}
%& \left< \dfrac{\gamma_1}{z_1-\psi_1}, %\cdots,\dfrac{\gamma_N}{z_N-\psi_N},\gamma_{N+1},\cdots,\gamma_{N+M}\right>_{0,M+N,0}^{W\mathbb{P}(m,r),T} \\  = &  %\dfrac1{z_1\cdots z_N}\big( \dfrac1{z_1}+\cdots \dfrac1{z_N}\big)^{N+M-3} \dint_{W\mathbb{P}(m,r)} %\gamma_1\cup\cdots\cup\gamma_{N+M}.
%\end{split}
%\]
Let $\mathbf{t}=tH+\sum_{i=-m+1}^{i=r-1} t^i\mathbf{1}_{i}$, where $\{\mathbf{1}_i,H\}$ form a basis of $H_\TT^\ast(\mathcal{F}(m,r),\mathbb{C})$ as a $\CC[\mathbf{w}]$-module.
Suppose that $2g-2+N+M>0$ or $N>0$. Given $\gamma_1,\cdots,\gamma_{M+N}\in H_{\CR,\TT} ^\ast (\mathcal{F}(m,r))$, we define
\[
\begin{split}
& \left<\!\left< \dfrac{\gamma_1}{z_1-\hat{\psi}_1}, \cdots,\dfrac{\gamma_N}{z_N-\hat{\psi}_N},\gamma_{N+1},\cdots,\gamma_{N+M}\right>\!\right>_{g,N+M}^{\mathcal{F}(m,r),\TT} \\  = &  \dsum_{d\geq 0}\dsum_{\ell\geq 0} \dfrac{Q^d}{\ell!}  \left< \dfrac{\gamma_1}{z_1-\hat{\psi}_1}, \cdots,\dfrac{\gamma_N}{z_N-\hat{\psi}_N},\gamma_{N+1},\cdots,\gamma_{N+M},\underbrace{\mathbf{t},\cdots,\mathbf{t}}_{\ell\;\mathrm{ times }}\right>_{g,M+N+\ell,d}^{\mathcal{F}(m,r),\TT}.
\end{split}
\]

Let
\[
\begin{split}
\mathbf{u_j}&=\dsum_{a\geq 0} (u_j)_a z^a,\\
%u_a&=\dsum_{\alpha=1}^{m+r} u_a^\alpha \phi_\alpha(q),\\
F_{g,N}^{\mathcal{F}(m,r),\TT}(\bu_1,\cdots ,\bu_N,\bt)&=\dsum_{a_1,\cdots,a_N\geq 0}\left<\!\left<\tau_{a_1}((u_1)_{a_1}),\cdots,\tau_{a_N}((u_N)_{a_N})\right>\!\right>_{g,N}^{\mathcal{F}(m,r),\TT}.
\end{split}
\]

For fixed $M,N\in\mathbb{Z}_{\geq 0}$, consider $\pi:\Mbar_{g,N+M}(\mathcal{F}(m,r),d)\to \Mbar_{g,N}$ which forgets the map to $\cI \cF(m,r)$ and the last $M$ marked points and stabilizes it.
Let $\bar{\mathbb{L}}_i$ be the pull-back of $\mathbb{L}_i$ along $\pi$.
Let $\bar{\psi}_i=\pi^\ast(\psi_i)=c_1(\bar{\mathbb{L}}_i)$ be the pull-back of $\psi$-classes on $\Mbar_{g,N}$. Define
\begin{equation*}
\bar{F}_{g,N}^{\cF(m,r),\TT}(\bu_1,\cdots, \bu_N,\bt) =  \dsum_{M,d,a_1,\cdots,a_N\geq 0} \dfrac{Q^d}{M!N!}\cdot\dint_{[\Mbar_{g,M+N}(\mathcal{F}(m,r),d)]^{w}} \dprod_{j=1}^N \mathrm{ev}_j^\ast((u_j)_{a_j})·(\bar{\psi}_j^{\bT})^{a_j} \dprod_{i=1}^M \mathrm{ev}_{i+N}^\ast(\mathbf{t}).
\end{equation*}
Let $F_{g,N}^{\mathcal{F}(m,r),\TT}(\mathbf{u,t}),\bar{F}_{g,N}^{\mathcal{F}(m,r),\TT}(\mathbf{u,t})$ be such that all $\mathbf{u_j}=\mathbf{u}$. We let $Q=1$ in the following context.

We define the total descendent potential and the ancestor potential of $\mathcal{F}(m,r)$ as follows, where $\hbar$ is a formal parameter:
\[
\begin{split}
&D^{\mathcal{F}(m,r),T}(\mathbf{u})=\exp\big( \sum_{N,g}\hbar^{g-1}F_{g,N}^{\mathcal{F}(m,r),T}(\mathbf{u,0})\big),\\
&A^{\mathcal{F}(m,r),\TT}(\mathbf{u,t})=\exp\big(\sum _{N,g}\hbar^{g-1}\bar{F}_{g,N}^{\mathcal{F}(m,r),\TT}(\mathbf{u,t}) \big).
\end{split}
\]

In fact, by Givental's formula \cite{Given}, we have
\[D^{\cF(m,r),\TT}(\mathbf{u})=\exp(F_1^{\mathcal{F}(m,r),\TT}) \hat{\mathcal{S}}^{-1} A^{\mathcal{F}(m,r),\TT}(\mathbf{u,t}),\] where $F_1^{\mathcal{F}(m,r),\TT}=\sum_N F_{1,N}^{\mathcal{F}(m,r),\TT}$.
We shall give an equivalent graph sum formula in Section \ref{A-model graph sum}.

\subsection{Mirror theorem of $\cF(m,r)$ as a toric stack}
In this section we follow the definition of I-function in \cite{Fang20}. \color{black}

Let $I_0=\bigcap_{I'\in\cA} I$ and let  $\cA'=\left\{ I'-I_0\mid  I'\in\mathcal{A}\right\}$. In our case, $I_0=\{-m+1,-m+2,\dots,r-1\}$ and $\cA'=\{\{-m\},\{r\},\{-m,r\}\}.$
Let $\overline{D}_i$ be the image of $D_i$ in 
$$
\LL^{\vee}/\sum_{i\in I_0}\ZZ D_i\cong H^2(\cF(m,r),\ZZ).
$$
Then $\overline{D}_i=0$ for $i\in I_0$.  
Let  $\{e_a\}_{a\in [-m,r]\setminus\{-1\}}$  be the integral basis of $\mathbb{L}^{\vee}$ defined in 
Section \ref{sec:fan}, 
% such that $e_a\in \mathrm{cl}(\tilde{C}_\mathcal{X})$ (the closure of $\tilde{C}_\mathcal{X}$) for all $a$
and let $\bar{e}_a$ be the image of $e_a$ in $H^2(\cF(m,r);\ZZ)$. Recall that $D_a = e_a$ for $a\in [-m,r]\setminus \{-1\}$,
and $D_{-1}= \sum_{i\in [-m,r]\setminus\{-1\}} ie_i$, so 
$H^2(\cX;\ZZ)$ 
% Then it is straightforward to show that $e_a\in\sum \mathbb{Z}D_i$ implies $\bar{e}_a=0$.
% Let $m_{ia}$ be a matrix such that $D_i=\sum_{a} m_{ia}e_a$, where $m_{ia}\in\mathbb{Z}$. Note that $\overline{D}_i=\sum_a m_{ia}\bar{e}_a$, and that $\overline{D}_i=0$ if $i\in I_0$.
is generated by $\overline{D}_{-m}=\bar{e}_{-m}$ and $\overline{D}_{r}=\bar{e}_{r}$ with the relation 
$-m \bar{e}_{-m}+r\bar{e}_{r}=\overline{D}_{-1}=0$,
 where $\bar{e}_{-m}$ corresponds to $y_2^m=\frac{1}{m}H$ and $\bar{e}_{r}$ corresponds to $y_1^{r}=\frac{1}{r}H.$

Let $\mathbb{K}=\left\{ d\in\mathbb{L}\otimes\mathbb{Q} : \{i:\left<D_i,d\right>\in\mathbb{Z}\}\in\mathcal{A} \right\}$ and
$\KKeff=\left\{ d\in\mathbb{L}\otimes\mathbb{Q} : \{i:\left<D_i,d\right>\in\mathbb{Z}_{\geq 0} \}\in\mathcal{A} \right\}$.  In our case, $\mathbb{K}=\mathbb{K}_{-m}\cup\mathbb{K}_{r}$ with
$$\mathbb{K}_{-m}=\left\{ d=(c_{-m},c_{-m+1},\dots,c_{-2},c_0,\dots,c_{r})\in\mathbb{L}\otimes\mathbb{Q} :c_{-m}\in\frac{1}{m}\ZZ, and \ c_{i}\in\ZZ \ for \ i\neq-m \right\}$$
$$\mathbb{K}_{r}=\left\{ d=(c_{-m},c_{-m+1},\dots,c_{-2},c_0,\dots,c_{r})\in\mathbb{L}\otimes\mathbb{Q} :\ c_{r}\in\frac{1}{r}\ZZ, and \ c_{i}\in\ZZ \ for \ i\neq r \right\}$$
And $\KKeff=\KKeff_{-m}\cup\KKeff_{r}$ with
$$\KKeff_{-m}=\left\{ d=(c_{-m},c_{-m+1},\dots,c_{-2},c_0,\dots,c_{r})\in \mathbb{K}_{-m}: \ c_{i}\in\ZZ_{\geq0} \ for \ d\neq-m , \ \langle D_{-1},d\rangle\geq0\right\}$$
$$\KKeff_{r}=\left\{ d=(c_{-m},c_{-m+1},\dots,c_{-2},c_0,\dots,c_{r})\in \mathbb{K}_{r}:\ c_{i}\in\ZZ_{\geq0} \ for \ i\neq r, \ \langle D_{-1},d\rangle\geq0\right\}$$

We choose a rational basis $\{p_{a}\}_{a\in I_0\cup\{*\}}$ of $\LL^{\vee}\otimes\QQ$ such that
$$\bigoplus_{a}\RR_{\geq0}p_{a}=(\bigoplus_{i\in I\setminus\{-m\}}\RR_{\geq0}D_i)\bigcap(\bigoplus_{i\in I\setminus\{r\}}\RR_{\geq0}D_i).$$

In our case, let $p_{a}=D_{a}$ for $a\in I_{0}$ and $p_{*}=\dsum_{i=1}^{r}iD_{i}=\dsum_{i=1}^{m}iD_{-i}$. Then we can write $D_{i}$ in terms of $p_{a}$:
\begin{align} 
D_{-m}&=\frac{1}{m}p_{*}-\dsum_{i=1}^{m-1}\frac{i}{m}p_{-i}\\
D_{i}&=p_{i} \ for \ i\in I_{0}\\
D_{r}&=\frac{1}{r}p_{*}-\dsum_{i=1}^{r-1}\frac{i}{r}p_{i}    
\end{align}

Let $\hat{\rho}=\dsum_{i\in I}D_{i}=\dsum_{i=0}^{m-1}(1-\frac{i}{m})p_{-i}+\dsum_{i=1}^{r-1}(1-\frac{i}{r})p_{i}+(\frac{1}{m}+\frac{1}{r})p_{*}.$

Let the map $v:\mathbb{K}\to\mathbb{K}/\LL\simeq\mathrm{Box}(\Sigma_{\omega})$ defined by:
$v(d)=\dsum_{i=-m}^{r}\{-\langle D_{i},d\rangle\}b_{i}=-m\{-d_{-m}\}+r\{-d_{r}\}.$ In our case, $v(d)=\{-\langle D_{-m},d\rangle\}b_{-m}+\{-\langle D_{r},d\rangle\}b_{r}$. We also write $v(d)=-m\{\langle-\hat\rho,d\rangle\}$ for $d\in\KK_{-m}$ and $v(d)=r\{\langle-\hat\rho,d\rangle\}$
for $d\in\KK_{r}$.

Let $q_a$ be coordinates on $\mathrm{Hom}(\mathbb{L},\mathbb{C}^*)$ with respect to the basis of $\{p_a \}_{I_{0}\cup\{*\}}$. Let $\deg q_{a}=\langle\hat{\rho},p_{a}\rangle$. In our case, $\deg q_{-l}=1-\frac{l}{m}$ for $0\leq l\leq m-1$, $\deg q_{l}=1-\frac{l}{r}$ for $0\leq l\leq r-1$, and $\deg q_{*}=\frac{m+r}{mr}$.

For convenience, we specially define $\tilde{d}_{-m}=\langle D_{-m},d\rangle=\frac{1}{m}(d_{*}-\dsum_{i=1}^{m-1}id_{-i})$, and $\tilde{d}_{r}=\langle D_{r},d\rangle=\frac{1}{r}(d_{*}-\dsum_{i=1}^{r-1}id_{i})$.

We rewrite the constraints of $d\in \KKeff$ in terms of $\{d_{a}\}$. $\langle D_{i},d\rangle\in\ZZ_{\geq0}$ for $i\in I_{0}$ is equivalent to $d_{i}\in\ZZ_{\geq 0}$. For $d\in\KKeff_{-m}$, $\langle D_{-m},d\rangle\in\ZZ_{\geq0}$ is equivalent to $\tilde{d}_{-m}\in\ZZ_{\geq 0}$. For $d\in\KKeff_{r}$, $\langle D_{r},d\rangle\in\ZZ_{\geq0}$ is equivalent to $\tilde{d}_{r}\in\ZZ_{\geq 0}$. Note that $d_{*}=m\tilde{d}_{-m}+\dsum_{i=1}^{m-1}id_{-i}=r\tilde{d}_{r}+\dsum_{i=1}^{r-1}id_{i}$
\begin{Def} \cite{Iritani} 
The $I$-function of $\mathcal{X}$ is an $H_{CR}^\ast (\mathcal{X})$-valued power series defined by
\[
I(q,z)=e^{(\sum_\alpha \bar{p}_a \log q_a)/z} \sum_{d\in\KKeff} q^d \cdot \dfrac{\prod_{\nu,i:\left<D_i,d\right>\leq \nu<0}(\overline{D}_i+(\left<D_i,d\right>-\nu)z)}{\prod_{\nu,i:\left<D_i,d\right>> \nu\geq 0}(\overline{D}_i+(\left<D_i,d\right>-\nu)z)}\mathbf{1}_{v(d)},
\]
where $q^d=\prod_a q_a^{d_a}$.
\end{Def}

Now we calculate the $I$-function of $\cF(m,r)$  by calculating the coefficient of $q^d=\prod_a q_a^{d_a}$ term with direct calculation:
\begin{equation}
    I(q,z)[q^{d}]=\frac{e^{(H\log q_{*})/z}}{z^{\sum_{i\in I}\lceil \langle D_{i},d\rangle\rceil}\cdot\dprod_{i\in I_{0}}d_{i}!}\cdot\frac{\Gamma(\frac{H}{mz}+1-{\{-\tilde{d}_{-m}}\})}{\Gamma(\frac{H}{mz}+\tilde{d}_{-m}+1)}\cdot\frac{\Gamma(\frac{H}{rz}+1-   {\{-\tilde{d}_{r}}\})}{\Gamma(\frac{H}{rz}+\tilde{d}_{r}+1)}\mathbf{1}_{v(d)}
\end{equation}
for $d\in\KKeff$. The $\Gamma$-function here is defined formally. 
\begin{Def} \cite{Fang20} 
The $\mathbb{T}$-equivariant $I$-function of $\mathcal{X}$ is an $H_{\mathbb{T},CR}^\ast (\mathcal{X})$-valued power series defined by
\[
I^{\mathbb{T}}(q,z)=e^{(\sum_\alpha \bar{p}_{a}^{\mathbb{T}} \log q_a)/z} \sum_{d\in\KKeff} q^d \cdot \dfrac{\prod_{\nu,i:\left<D_i,d\right>\leq \nu<0}(\overline{D}_i^{\mathbb{T}}+(\left<D_i,d\right>-\nu)z)}{\prod_{\nu,i:\left<D_i,d\right>> \nu\geq 0}(\overline{D}_i^{\mathbb{T}}+(\left<D_i,d\right>-\nu)z)}\mathbf{1}_{v(d)},
\]
where $q^d=\prod_a q_a^{d_a}$.
\end{Def}

In our case, $\overline{D}_{i}^{\mathbb{T}}=0$ for $i\neq-m,r$. $\overline{D}_{-m}^{\mathbb{T}}=\frac{1}{m}y_{2}^{m}$ and $\overline{D}_{r}^{\mathbb{T}}=\frac{1}{r}y_{1}^{r}$. And according to Equation 74 of \cite{Iritani}, we have $\bar{p}_{a}^{\mathbb{T}}=\overline{D}_{a}^{\mathbb{T}}-\mathbf{w}_{a}=-\bold{w}_{a}$ for $a\in I_0$, and $\bar{p}_{*}=ry_1^{r}-\bold{p}_{+}=my_{2}^{m}-\bold{p}_{-}$. We define $H:=\bar{p}_{*}\in H^{2}_{CR,\mathbb{T}}(\cF(m,r),\CC)$. So we have $e^{(\sum_a \bar{p}_{a}^{\mathbb{T}} \log q_a)/z}=e^{(H\log q_{*}-\sum_{a\in I_{0}}\mathbf{w}_{a}\log q_{a})/z}.$
\begin{equation}\label{equivariant I-function}
    I^{\mathbb{T}}(q,z)[q^{d}]=\frac{e^{(H\log q_{*}-\sum_{a\in I_{0}}\mathbf{w}_{a}\log q_{a})/z}}{z^{\sum_{i\in I}\lceil \langle D_{i},d\rangle\rceil}\cdot\dprod_{i\in I_{0}}d_{i}!}\cdot\frac{\Gamma(\frac{H+\bold{p}_{-}}{mz}+1-\{-\tilde{d}_{-m}\})}{\Gamma(\frac{H+\bold{p}_{-}}{mz}+\tilde{d}_{-m}+1)}\cdot\frac{\Gamma(\frac{H+\bold{p}_{+}}{rz}+1-{\{-\tilde{d}_{r}}\})}{\Gamma(\frac{H+\bold{p}_{+}}{rz}+\tilde{d}_{r}+1)}\mathbf{1}_{v(d)}
\end{equation}
for $d\in\KKeff$.

Let
\[\begin{split}
&\partial_a=q_a\frac{\partial}{\partial q_a}, \mathcal{D}_i=\sum_a m_{ia}\cdot z\partial_a,\\
&\mathcal{P}_d=q^d \prod_{i,\nu:-\left<D_i,d\right>>\nu\geq 0} (\mathcal{D}_i-\nu z)- \prod_{i,\nu:\left<D_i,d\right>>\nu\geq 0} (\mathcal{D}_i-\nu z).
\end{split}
\]
By direct calculations (Lemma 4.6 in \cite{Iritani}) we have:
\begin{lemma}
It holds that $\mathcal{P}_d(I(q,z))=0,\;\forall d\in\mathbb{N}$.
\end{lemma}

Let $\mathrm{Eff}_\mathcal{X}\subseteq H_2(X,\mathbb{Z})$ be the semigroup generated by effective stable maps.
For an arbitrary $\tau\in H_{CR}^\ast(\mathcal{X})$, let $\tau_{0,2}$ be the component of $\tau$ consisting of terms of degree 2 and degree 0, and $\tau'=\tau-\tau_{0,2}$.
Choose an arbitrary basis $\{H_a\}$ of $H_{CR}^\ast(\mathcal{X})$ as a $\mathbb{C}$-module, and $\{H^a\}$ its dual basis. Then $H^a$ could be regarded as in $H_{CR}^\ast(\mathcal{X})$ by applying the Poincare duality.

\begin{Def}
The $J$-function
\[J(\tau,z)=e^{\tau_{0,2}/z} \big(\mathbf{1}+\sum_{\stackrel{(d,\ell)\neq (0,0)}{d\in{\mathrm{Eff}_\mathcal{X}}}}\sum_a \dfrac1{\ell!}\left<\mathbf{1},\tau',\cdots,\tau',\frac{H_a}{z-\psi}\right>^X_{0,\ell+2,d} \cdot e^{\tau_{0,2},d}H^a \big).\]
\end{Def}

Let $I(q,z)=1+\frac{\tau(q)}z+o(z^{-1})$ be the expansion of $I$ with respect to $z$ at $z=\infty$. Here $\tau$ is called the mirror map.

\begin{Def}
The equivarinat $J$-function
\[J^{\mathbb{T}}(\tau,z)=e^{\tau_{0,2}/z} \big(\mathbf{1}+\sum_{\stackrel{(d,\ell)\neq (0,0)}{d\in{\mathrm{Eff}_\mathcal{X}}}}\sum_a \dfrac1{\ell!}\left<\mathbf{1},\tau',\cdots,\tau',\frac{H_a}{z-\psi}\right>^X_{0,\ell+2,d} \cdot e^{\tau_{0,2},d}H^a \big).\]
\end{Def}
We also similarly let $I^{\mathbb{T}}(q,z)=1+\frac{\tau(q)}z+o(z^{-1})$ be the expansion of $I$ with respect to $z$ at $z=\infty$. Here $\tau$ is called the equivariant mirror map.
\begin{thms}\label{non-equivariant mirror theorem}[Mirror Theorem]It holds that $I(q,z)=J(\tau(q),z)$.
\end{thms}
\begin{thms}\label{equivariant mirror theorem}[Equivariant Mirror Theorem]
$I^{\mathbb{T}}(q,z)=J^{\mathbb{T}}(\tau(q),z)$.
\end{thms}
\begin{remark} \label{Equivariant $J$-function in MT11}
    Mirror theorem for weighted projective space was proved in \cite{CCIT15}. The general case was proved in \cite{CCLT09}. This theorem also implies the results of the equivariant $J$-function in Proposition A.8, A.9 of \cite{MT11}, by taking the $q_{i}\to0$, and $\w_{i}\to0$ limit for $i\in I_{0}$, under the following correspondence of notations:
    \begin{eqnarray*}
        Qe^{t}&=q_{*}\\
        \nu&=\frac{\bold{p}}{r}\\
        \bar{\nu}&=\frac{-\bold{p}}{m}\\       \bold{1}_{0/r}&=\frac{H+\bold{p}_{+}}{\bold{p}}\\        \bold{1}_{0/m}&=\frac{H+\bold{p}_{-}}{-\bold{p}}. 
    \end{eqnarray*}
\end{remark}
And we also note that we can also write the equivariant $I$-function in a similar form by setting a slightly different notation. We define:
\begin{eqnarray*}    \bold{1}_{0,r}&=\frac{H+\bold{p}_{+}}{\bold{p}}\\
\bold{1}_{0,m}&=\frac{H+\bold{p}_{-}}{-\bold{p}}.    
\end{eqnarray*}
And $\bold{1}_{i,m}=\bold{1}_{i}$, $\bold{1}_{i,r}=\bold{1}_{i}$ for $i\neq 0$. Then an alternative form of equivariant $I$-function can be written as:
\begin{eqnarray*}
     &I^{\mathbb{T}}(q,z)=e^{(-\sum_{a\in I_{0}}\mathbf{w}_{a}\log q_{a})/z}(\dsum_{d\in\KKeff_{r}}e^{\frac{-\bold{p}_{-}\log q_{*}}{z}}q^{d}{z^{-\sum_{i\in I}\lceil \langle D_{i},d\rangle\rceil}}\cdot\frac{\Gamma(\frac{\bold{p}}{rz}+1-{\{-\tilde{d}_{r}}\})}{\tilde{d}_{-m}!\dprod_{i\in I_{0}}d_{i}!\Gamma(\frac{\bold{p}}{rz}+\tilde{d}_{r}+1)}\mathbf{1}_{v(d),r}\\&+\dsum_{d\in\KKeff_{-m}}e^{\frac{-\bold{p}_{+}\log q_{*}}{z}}q^{d}{z^{-\sum_{i\in I}\lceil \langle D_{i},d\rangle\rceil}}\cdot\frac{\Gamma(\frac{-\bold{p}}{mz}+1-\{-\tilde{d}_{-m}\})}{\tilde{d}_{r}!\dprod_{i\in I_{0}}d_{i}!\Gamma(\frac{-\bold{p}}{mz}+\tilde{d}_{-m}+1)}\mathbf{1}_{v(d),m}).
\end{eqnarray*}
\begin{remark}
    This alternative form of equivariant $I$-function splits the terms of $\KKeff_{r}$ and $\KKeff_{-m}$, also involving only $\CC[\w]$-coefficient. Note that this form can also take the non-equivariant limit to get the non-equivariant $I$-function, using the following relations:
    \begin{eqnarray*}
        \bold{1}_{0,r}+\bold{1}_{0,m}&=\bold{1}\\
        \bold{p}_{-}\bold{1}_{0,r}+\bold{p}_{+}\bold{1}_{0,m}&=-H
    \end{eqnarray*}
    and for $i\geq2$, and $\w_{i}\to0$
    \begin{equation*}        \bold{p}_{-}^{i}\bold{1}_{0,r}+\bold{p}_{+}^{i}\bold{1}_{0,m}=0.
    \end{equation*}
\end{remark}

Now we calculate the equivariant mirror map $\tau(q)$. Note that the equivariant mirror map $\tau(q)$ also gives non-equivariant mirror map by taking non-equivariant limit. 

Now we calculate the equivariant mirror map $\tau(q)$. We use the notations of Lemma 4.2 in \cite{Iritani} to formulate our $\tau(q)$.

The first term of $\tau(q)$ in our case, which only corresponds to $\tau_{0,2}$, is $\tau_{0,2}=H\log q_{*}-\sum_{a\in I_{0}}\mathbf{w}_{a}\log q_{a}$. This also shows the $H$-coefficient $t=\log q_{*}$ 

And the second and third terms are given by the contribution of $d\in\KKeff$ such that $0<\langle \hat\rho,d\rangle\leq1$. We also notice that for $\tilde{d}_{-m}\in\ZZ_{<0}$ or $\tilde{d}_{r}\in\ZZ_{<0}$, there are extra factors of $z^{-1}$, which will not be  contained in the mirror map.

We write $\tau=\tau_{0,2}+\dsum_{i\in I_{0}\setminus\{0\}}t^{i}\mathbf{1}_{i}$, and calculate the coefficient of $\mathbf{1}_{i}$ respectively. The following lemma shows that each of $\mathbf{1}_{i}$ component consists of $q^{d}$-terms with same degree.
\begin{lemma}
   For $d\in\KKeff$ such that $0<\langle \hat\rho,d\rangle\leq1$. We have:
   \begin{align*}
       v(d)&=-m(1-\langle\hat{\rho},d\rangle)=-m(1-\deg d)\\
       v(d)&=r(1-\langle\hat{\rho},d\rangle)=r(1-\deg d)
   \end{align*}

\end{lemma}
\begin{proof}
    This comes from direct calculation.
\end{proof} 
From this lemma we know each $t^{i}$ is a weighted homogeneous polynomial $\{q_{a}\}$. 

We write $q^{d}=q_{*}^{m\tilde{d}_{-m}}\dprod_{i=1}^{m-1}(q_{-i}q_{*}^{i})^{d_{-i}}\dprod_{i=0}^{r-1}q_{i}^{d_{i}}$ for $d\in\KKeff_{r}$, and $q^{d}=q_{*}^{r\tilde{d}_{r}}\dprod_{i=1}^{r-1}(q_{i}q_{*}^{i})^{d_{i}}\dprod_{i=0}^{m-1}q_{-i}^{d_{-i}}$ for $d\in\KKeff_{-m}$. Considering the degree constraints, For $0<\deg d\leq1$, the first two parts of the product vanishes. So for $d\in\KKeff_{r}$, $d$ reduces to $\dprod_{i=0}^{r-1}q_{i}^{d_{i}}$, with $\tilde{d}_{r}=\deg d-\dsum d_{i}$. For $d\in\KKeff_{-m}$, $\dprod_{i=0}^{m-1}q_{-i}^{d_{i}}$  with $\tilde{d}_{-m}=\deg d-\dsum d_{i}$.
\begin{lemma}[Equivariant mirror maps]
    \begin{equation}
        t^{i}=\dsum_{\sum_{j=1}^{m-1}(m-j)d_{-j}=m+i}\frac{q^{d}}{\dprod_{j=1}^{m-1}d_{-j}!}\cdot\frac{\Gamma(1+\frac{i}{m})}{\Gamma(2+\frac{i}{m}-\dsum_{j=1}^{m-1}{d_{-j}})}
    \end{equation}
    for $-m+1\leq i\leq -1$, and
    \begin{equation}
        t^{i}=\dsum_{\sum_{j=1}^{r-1}(r-j)d_{j}=r-i}\frac{q^{d}}{\dprod_{j=0}^{r-1}d_{j}!}\cdot\frac{\Gamma(1-\frac{i}{r})}{\Gamma(2-\frac{i}{r}-\dsum_{j=1}^{r-1}{d_{j}})}
    \end{equation}
    for $1\leq i\leq r-1$, and
    \begin{equation}
        t^{0}=q_{0}-\dsum_{i\in I_{0}}\bold{w}_{a}\log q_{a}
    \end{equation}
\end{lemma}
\begin{proof}
    The lemma comes dirctly from taking limit $z\rightarrow\infty$ for $zI(q,z)[q^{d}]$.
\end{proof}
With this lemma we can conclude the equivariant mirror map $\tau(q)=tH+\dsum_{i\in I_{0}}t^{i}\bold{1}_{i}$
Similarly, we define
\[\begin{split}
&\partial_a=q_a\frac{\partial}{\partial q_a}, \mathcal{D}_i^{\mathbb{T}}=\sum_a m_{ia}\cdot z\partial_a+\mathbf{w}_{i},\\
&\mathcal{P}^{\mathbb{T}}_d=q^d \prod_{i,\nu:-\left<D_i,d\right>>\nu\geq 0} (\mathcal{D}_i^{\mathbb{T}}-\nu z)- \prod_{i,\nu:\left<D_i,d\right>>\nu\geq 0} (\mathcal{D}_i^{\mathbb{T}}-\nu z).
\end{split}
\]
\begin{lemma}\label{P_d equivariant}
It holds that $\mathcal{P}^{\mathbb{T}}_d(I^{\mathbb{T}}(q,z))=0,\;\forall d\in\mathbb{L}$.
\end{lemma}
\begin{proof}
    By direct calculation.
\end{proof}

\subsection{Equivariant quantum cohomology}
\

%\noindent We have actually the following expression
%\[
%\hspace{22mm}-r+1 \hspace{30mm} 0 \hspace{30mm} m-1\hspace{19mm}
%\]
%\vspace{-8mm}
%\[
%\left[
%\begin{array}{l}
%D_{-m}\\D_{-m+1}\\ \vdots\\D_{-1}\\D_1\\ \vdots \\ D_{r-1}\\ D_{r}
%\end{array}
%\right]
%=\left[
%\begin{array}{ccccccccccc}
%1 &  & 2 & \cdots & r-1 & n & 0  & \cdots & 0 & & 0 \\
%0 & & 0 & \cdots & 0   & 0 & n &  \cdots & 0 & & 0 \\
%\vdots&  & \vdots & & \vdots & \vdots & \vdots & & \vdots &&  \vdots \\
%0 & & 0 & \cdots & 0   & 0 & 0 & \cdots & 0 &&  n \\
%m &&  0 & \cdots & 0   & 0 & 0 & \cdots & 0 & & 0 \\
%\vdots & & \vdots & & \vdots & \vdots & \vdots & & \vdots &&  \vdots \\
%0 & & 0 & \cdots & m   & 0 & 0 &  \cdots & 0 &&  0 \\
%0 & & 0 & \cdots & 0   & m & m-1 & \cdots & 2 & & 1 \\
%\end{array}
%\right]
%\left[
%\begin{array}{l}
%e_{-r+1}\\ \vdots \\ \vdots\\ \vdots\\ \vdots\\ \vdots \\  \vdots\\ e_{m-1}
%\end{array}
%\right],
%\]
In this section we calculate equivariant quantum cohomology ring  $QH^{*}_{\mathbb{T}}(\cF(m,r),\CC)$ using the relation $(\mathcal{P}_{d}^{\mathbb{T}}J,H_{\sigma})=0$ for any $H_{\sigma}$. We start from the calculation of non-equivariant operator $\calD_{i}$.

$J(\tau,z)=\sum_a  \left<\!\left<1,\frac{H_a}{z-\psi} \right>\!\right>_{0,2} H^a$ follows directly from Definition 2.2.

On the other hand, consider the Frobenius algebra $V=QH^\ast(\cF(m,r),\CC)\cong\mathrm{Jac}(\cF(m,r))$.
Let $H_\alpha$ (which coincides with the $H_a$ above in our case) be its basis as a $\mathbb{C}$ vector space. We identify $V$ and its tangent space $T_pV$ for $p$ a semisimple point on $V$.
Define quantum connection $\nabla_a=z\frac{\partial}{\partial t_{a}} -\mathbf{1}_a\ast$ for $a\in I_{0}$ and $\nabla_*=z\frac{\partial}{\partial t} -H\ast$.
Consider differential equations system: (Quantum Differential Equation, QDE)
\[
\nabla_a h=0,\quad a\in I_{0}\bigcup\{*\}
\]
Let $S_\sigma=\sum_a \left<\!\left<H_a,\frac{H_\sigma}{z-\psi} \right>\!\right> H^a$. This forms a set of fundamental solutions to the QDE.

Let $(\alpha,\beta)$ be the Poincare paring $\int_{\cF(m,r)} \alpha \cup \beta$. Then
$(J,H_\sigma) = \left<\!\left<1,\frac{H_\sigma}{z-\psi} \right>\!\right>_{0,2}^{\cF(m,r)} = (S_\sigma,1)$
by the duality of $H_a$ and $H^a$.
Inducting on $k$ with the Leibniz rule, we have
$\big( z\frac{\partial}{\partial t_{i_1}} \cdots z\frac{\partial}{\partial t_{i_k}} J, H_\sigma \big) = (H_{i_1}\cdots H_{i_k}S_\sigma,1)$. Then we use the following fact to calculate $\calD_{i}$:
\begin{equation}
    \calD_{i}=\dsum_{a}m_{ia}zq_{a}\frac{\partial}{\partial q_{a}}=\dsum_{a}m_{ia}q_{a}\sum_{t^{i}}\frac{\partial t^{i}}{\partial q_{a}}\frac{\partial}{\partial t^{i}}
\end{equation}
In our case the basis $H_{i}$ is taken in respect of our $t^{i}$. We can notice $\frac{\partial}{\partial q_{0}}=\frac{\partial}{\partial t^{0}}$, $\frac{\partial}{\partial q_{1}}=\frac{\partial}{\partial t^{1}}$, and $\frac{\partial}{\partial q_{-1}}=\frac{\partial}{\partial t^{-1}}$

And by direct calculation, we have:
\begin{equation}
    \prod_{i,\nu=1,2,\dots,r} (\mathcal{D}_1^{\mathbb{T}}-\nu z)J^{\TT}=q_{1}^{r}(z\frac{\partial}{\partial q_{1}}+\bold{w}_{1})^{r}J^{\TT}
\end{equation}
This implies:
\begin{equation}
    \prod_{\nu=0}^{l-1} (\mathcal{D}_1^{\mathbb{T}}-\nu z)S_{\sigma}=q_{1}^{l}(\bold{1}_{1}*)^{l}S_{\sigma}
\end{equation}
for any $S_{\sigma}$.
Similarly, we have:
\begin{equation}
    \prod_{\nu=0}^{l-1}(\mathcal{D}_{-1}^{\mathbb{T}}-\nu z)S_{\sigma}=q_{-1}^{l}(\bold{1}_{-1}*)^{l}S_{\sigma}
\end{equation}
for any $S_{\sigma}$.

Using the Lemma \ref{P_d equivariant}, we take $d$ by setting the values of all $\langle D_{i},d\rangle$ with the constraints of $\dsum_{i\in I}iD_{i}=0$. By setting $\langle D_{1},d\rangle=i$ and $\langle D_{i},d\rangle=-1$, with other pairings 0, the lemma shows:
\begin{equation}
    \mathcal{D}_{i}^{\TT}J^{\TT}=q_{i}(z\frac{\partial}{\partial q_{1}}+\bold{w}_{1})^{i}J^{\TT}
\end{equation}
 for $i=1,2,\dots,r$, and any $S_{\sigma}$. Similarly, we also have:
\begin{equation}
    \mathcal{D}_{-i}^{\TT}J^{\TT}=q_{-i}(z\frac{\partial}{\partial q_{-1}}+\bold{w}_{-1})^{i}J^{\TT}
\end{equation}
 for $i=1,2,\dots,m$, and any $S_{\sigma}$. By setting $\langle D_{1},d\rangle=1$ and $\langle D_{-1},d\rangle=1$, with other pairings 0, the lemma gives
\begin{equation}
    \bold{1}_{1}*\bold{1}_{-1}=q_{*}
\end{equation}

 For $\mathcal{D}_{r}^{\TT}$ and $\mathcal{D}_{-m}^{\TT}$, with direct calculation we also have: 
\begin{align*}
    \calD_{r}^{\mathbb{T}}J^{\TT}&=\frac{1}{r}((z\frac{\partial}{\partial t}+\bold{p}_{+})-\dsum_{i=1}^{r-1}i\calD_{i}^{\mathbb{T}})J^{\TT}\\
    \calD_{-m}^{\mathbb{T}}J^{\TT}&=\frac{1}{m}((z\frac{\partial}{\partial t}+\bold{p}_{-})-\dsum_{i=1}^{m-1}i\calD_{-i}^{\mathbb{T}})J^{\TT}.
\end{align*}
Now we use the following facts to calculate $(\bold{1}_{1}*)^{r}$, and similarly $(\bold{1}_{-1}*)^{m}$:
\begin{equation}
    [z^{0}]\dsum_{H_{a}}\big((z\frac{\partial}{\partial t_{1}}+\bold{w}_{1})^{r}J^{\TT},H^{a}\big)H_{a}=(\bold{1}_{1})^{r}
\end{equation}
and similarly,
\begin{equation}
    [z^{0}]\dsum_{H_{a}}\big((z\frac{\partial}{\partial t_{-1}}+\bold{w}_{-1})^{m}J^{\TT},H^{a}\big)H_{a}=(\bold{1}_{-1})^{m}
\end{equation}

By $\mathcal{P}_d I^{\mathbb{T}}=0, I^{\mathbb{T}}=J^{\mathbb{T}}$, so we know that:
\[
\begin{split}
0 & = \dsum_{H_{a}}(\mathcal{P}_d J^{\mathbb{T}}, H^a)H_a \\
& = \dsum_{H_{a}}\big( \big( q^d\displaystyle\prod_{i,\nu:-\left<D_i,d\right>>\nu\geq 0}(\calD^{\mathbb{T}}_{i}-\nu z) - \displaystyle\prod_{i,\nu:\left<D_i,d\right>>\nu\geq 0}(\calD^{\mathbb{T}}_{i} -\nu z)
\big)J,H^a\big)H_a
\end{split}
\]

Take $d$ letting $\langle D_{1},d\rangle=n$, and $\langle D_{r},d\rangle=-1$, we have:
\begin{equation}
    (\bold{1}_{1})^{r}=\frac{1}{r}(H+\bold{p}_{+}-\dsum_{i=1}^{r-1}iq_{i}(\bold{1}_{1})^i)
\end{equation}
Similarly, take $d$ letting $\langle D_{-1},d\rangle=m$, and $\langle D_{-m},d\rangle=-1$, we have:
\begin{equation}
    (\bold{1}_{-1})^m=\frac{1}{m}(H+\bold{p}_{-}-\dsum_{i=1}^{m-1}iq_{-i}(\bold{1}_{-1})^i)
\end{equation}
So if we take $X=\bold{1}_{1}$ and $\overline{X}=\bold{1}_{-1}$. By dimesion argument, $QH^{*}_{\mathbb{T}}(\cF(m,r),\CC)$ should have no other non-trivial relation. We obtain

\begin{equation}
    QH^\ast_{\mathbb{T}}(\cF(m,r),\CC)\cong\CC[X,\overline X][\bold{w}]/\langle X\overline X-q_{*},rX^{r}-m\overline{X}^{m}-\bold{p}+\dsum_{i=-m+1}^{-1}iq_{i}\overline{X}^{-i}+\dsum_{i=1}^{r-1}iq_{i}X^{i}\rangle
\end{equation}
From the above relation, we can also write $\overline{X}=q_{*}X^{-1}$, and if we set $\tilde{q}_{i}=q_{i}$ for $i=0,1,\dots,r$, and $\tilde{q}_{i}=q_{i}q_{*}^{-i}$ for $i=0,-1,\dots,-m$. Note that we have $\tilde{q}_{-m}=q_{*}^{m}$, and $\tilde{q}_{r}=1$. Then the equivariant quantum cohomology can be simplified into:
\begin{equation}
    QH^\ast_{\mathbb{T}}(\cF(m,r),\CC)\cong\CC[X][\bold{w}]/\langle -\bold{p}+\dsum_{i=-m}^{r}i\tilde{q}_{i}X^{i}\rangle
\end{equation}
With this, we directly conclude:

\begin{prop}
$QH^\ast_{\TT}(\cF(m,r),\CC)\cong\mathrm{Jac}(W_\mathbb{T})$ as Frobenius algebras.
\end{prop}
\begin{proof}
Note that it only remains to show the identification of the residue pairing and Poincare pairings, which comes from direct calculation by the residue formula.

However we include a proof which recovers the pairing from its non-equivariant limit as usual.

We identify $p$ with $\tilde{p}$, and $Y$ with $X$.

Taking non-equivariant limit and applying mirror theorem, we have the isomorphism of non-equivariant limits $QH^\ast(\cF(m,r))\cong\mathrm{Jac}(W)$ as Frobenius algebras. Here the non-equivariant limit is given by $w_i \rightarrow 0$ and $\tilde{w}_i \rightarrow 0$. We now prove that $w_i$ and $\tilde{w_i}$ affect neither the residue pairing nor the Poincare pairing with $1$, which directly leads to the result.

Let $(\cdot,\cdot)$ denote the non-equivariant pairing, $(\cdot,\cdot)_{\mathbb{T}}$ denote the equivariant pairing. For an representative element $g(t) \in \oplus_{j=-m}^{r}\mathbb{C}t^j$, we have the following results:

\noindent On the A-side,
\[(g(X),1)=\dint_{\wedge\mathcal{X}}g(X)=\dint_{\wedge\mathcal{X}^{\mathbb{T}}}g(X)=(g(X),1)_{\mathbb{T}};\]
On the B-side, let $f_{\mathbb{T}}(Y)=\partial W_{\mathbb{T}}/\partial y$, and $f(Y)=\partial W/\partial y$. We notice that $f_{\mathbb{T}}(Y)=f(Y)-\tilde{p}$. By the residue formula,
\[(g(Y),1)=-\Res_{Y\to\infty} \dfrac{g(Y)}{f(Y)}\dfrac{dY}{Y}-\Res_{Y\to0} \dfrac{g(Y)}{f(Y)}\dfrac{dY}{Y}=\dfrac{1}{2\pi}\lim _{R \rightarrow \infty} \dint_0^{2\pi}\dfrac{g(Re^{i\theta})}{f(Re^{i\theta})}d\theta-\dfrac{1}{2\pi}\lim _{R \rightarrow 0} \dint_0^{2\pi}\dfrac{g(Re^{i\theta})}{f(Re^{i\theta})}d\theta.\]
Hence we have
\[(g(Y),1)-(g(Y),1)_{\mathbb{T}}=-\dfrac{1}{2\pi }\lim _{R \rightarrow \infty} \dint_0^{2\pi}\dfrac{g(Re^{i\theta})\cdot \tilde{p}}{f(Re^{i\theta})(f(Re^{i\theta})-\tilde{p})}d\theta+\dfrac{1}{2\pi }\lim _{R \rightarrow 0} \dint_0^{2\pi}\dfrac{g(Re^{i\theta})\cdot \tilde{p}}{f(Re^{i\theta})(f(Re^{i\theta})-\tilde{p})}d\theta=0.\]
And as a result $(g_1(Y),g_2(Y))_{\mathbb{T}}=(g_1(X),g_2(X))_{\mathbb{T}}$.
\end{proof}
\subsection{Basis}
\

Take $z_{\alpha}(\alpha=0,\cdots,m+r-1)$ to be the roots of $rz^{r} +  \sum_{\ell=1}^{r-1}\ell \tilde{q}_\ell z^\ell - \sum_{\ell=1}^{m} \ell \tilde{q}_{-\ell} z^{-\ell} -\tilde{p}=0$ with respect to $z$.
Assume that $z_\alpha$ are distinct.
Let $\phi_\alpha=\prod_{\beta\neq \alpha}\frac{X-z_\beta}{z_\alpha-z_\beta}$. Then $\phi_\alpha$ is a canonical basis, i.e., $\phi_\alpha\cdot\phi_\beta=\delta_{\alpha\beta}\phi_\alpha$.

\begin{lemma}
Let $z_i$ be all roots of $f(z)$. Assume $z_i$'s are distinct. Then $\phi_j=\prod_{i\neq j}\frac{z-z_i}{z_j-z_i}$ is a representative of the canonical basis of $\mathbb{C}[z]/\left<f(z)\right>$.
\end{lemma}

\begin{proof}
Observing that the constructed $\phi_i$ is characterized by $\phi_i(z_j)=\delta_{ij}$, we obtain the lemma directly from the Lagrange interpolation formula.
\end{proof}

By direct calculations, we get $\Delta^\alpha=\frac{1}{(\phi_\alpha,\phi_\alpha)}=\frac{r\prod_{\beta\neq\alpha}(z_\alpha-z_\beta)}{z_\alpha^{m-1}}$.
Now consider several different bases for $QH_{\mathbb{T}}^\ast(\cF(m,r),\mathbb{C})$:
\begin{itemize}
  \item The natural basis $T_i=X^i$ and its dual basis $T^i$ with which $(T^i,T_j)=\delta^i_j$.
  \item The canonical basis $\phi_\alpha$ as defined above and its dual basis $\phi^\alpha=\Delta^\alpha(q)\phi_\alpha$.
  \item The normalized canonical basis $\hat{\phi}_\alpha=\sqrt{\Delta^\alpha(q)}\cdot\phi_\alpha$, and its dual basis $\hat{\phi}^\alpha=\hat{\phi}_\alpha$.
\end{itemize}
Regarding $\phi_\alpha$ as a function of $q$, we often write it as $\phi_\alpha(q)$, (same for $\hat{\phi}_\alpha,\phi^\alpha$ and $\hat{\phi}^\alpha$).
For an arbitrary point $pt\in QH_\TT^\ast(\cF(m,r),\mathbb{C})$, let $t^i,u^i,\bar{u}^i$ be coordinates such that
\[
pt=\sum t^iT_i=\sum u^i\phi_i(q)=\sum\bar{u}^i \phi_i(0).
\]
Regarding $t^i,u^i,\bar{u}^i$ as functions of $q$, we often write them as $t^i(q),u^\alpha(q),\bar{u}^\alpha(q)$. We often call $t^i(q)$ and $\bar{u}^\alpha(q)$ as flat coordinates, and $u^\alpha(q)$ as canonical coordinates. In our case, we set the flat basis $\{X^{i}\}$, with the coordinates $t^{i}(q)$ indexed by $i=0,1,\dots,m+r-1$.

We can explicitly write out the differential relation of $t^{i}(q)$ with $q_{i}$. Using the fact that:
\begin{equation*}
    \frac{\partial J}{\partial q_{i}}=\sum_a  \left<\!\left<1,X^{i},\frac{H_a}{z-\psi} \right>\!\right>_{0,2} H^a
\end{equation*}
for $i\geq 0$, and
\begin{equation*}
    \frac{\partial J}{\partial q_{i}}=\sum_a  \left<\!\left<1,q_{*}^{-i}X^{i},\frac{H_a}{z-\psi} \right>\!\right>_{0,2} H^a
\end{equation*}
for $i<0$.
Let $\bold{t}=\dsum_{i=0}^{m+r-1}t^{i}(q)X^{i}$, we have $\frac{\partial\bold{t}}{\partial q_{i}}=X^{i}$ for $i\geq 0$, and $\frac{\partial\bold{t}}{\partial q_{i}}=q_{*}^{-i}X^{i}$ for $i<0$.

\subsection{The A-model canonical coordinates and the $\Psi$-matrix}\label{sec:A-canonical}
In this subsection we calculate $\Psi$-matrix under the flat coordinates and canonical coordinates chosen in the previous subsection. 

By direct calculation, we have: 
\begin{eqnarray*}
\big(\frac{\partial}{\partial\bold{t}}\big)^{T}=&V(z_{0},z_{1},\dots,z_{m+r-1})\big(\frac{\partial}{\partial\bold{u}}\big)^{T}\\
\big({d\bold{t}}\big)^{T}=&V^{-1}(z_{0},z_{1},\dots,z_{m+r-1})^{T}\big({d\bold{u}}\big)^{T},
\end{eqnarray*}
where $V(z_{0},z_{1},\dots,z_{m+r-1})$ is the Vardermonde matrix with parameters $\{z_{0},z_{1},\dots,z_{m+r-1}\}$.
The above equations determine the canonical coordinates
$\bold{u}$ up to a constant in $\CC[\w]$.

We write $(\frac{\partial}{\partial\bold{t}})^{T}=V(z_{0},z_{1},\dots,z_{m+r-1})(\frac{\partial}{\partial\bold{u}})^{T}$. Then we have $({d\bold{t}})^{T}=V^{-1}(z_{0},z_{1},\dots,z_{m+r-1})^{T}({d\bold{u}})^{T}$.

For $\alpha\in \{0,1,\dots,m+r-1\}$ and $i\in \{0,1,\dots,m+r-1\}$, define $\Psi_i^{\,\ \alpha}$ by
$$
\frac{du^\alpha}{\sqrt{\Delta^\alpha(q)}} =\sum_{i=0}^{m+r-1} dt^i \Psi_i^{\,\ \alpha},
$$
and define the $\Psi$-matrix to be
$$
\Psi:= \left(\Psi_i^{\,\ \alpha} \right)_{i,\alpha}
$$
with row index $i$ and column index  $\alpha$. We can directly find out: 
$$\Psi=V(z_{0},z_{1},\dots,z_{m+r-1})\mathbf{diag}((\Delta^{0}(q))^{\frac{1}{2}},
\dots,(\Delta^{m+r-1}(q))^{\frac{1}{2}}).$$
This also explicitly gives  $\Psi_i^{\,\ \alpha}=z_{\alpha}^{i}(\Delta^{\alpha}(q))^{\frac{1}{2}}$.

Let $\Psi^{-1}$ be the inverse matrix of $\Psi$, so we have:
$$
\Psi^{-1}=\mathbf{diag}((\Delta^{0}(q))^{-\frac{1}{2}},
\dots,(\Delta^{m+r-1}(q))^{-\frac{1}{2}})V^{-1}(z_{0},z_{1},\dots,z_{m+r-1}).
$$
We can also explicitly write:
$$(\Psi^{-1})^{\,\,i}_{\alpha}:=(\Psi^{-1})_{\alpha,i}=(-1)^{i}(\frac{1}{r}z_{\alpha}^{m-1})(\Delta^{\alpha}(q))^{\frac{1}{2}}\dsum_{J\subset S\setminus\{\alpha\},|J|=m+r-i-1}z_{J}$$
where $z_{J}=\dprod_{\alpha\in J}z_{\alpha}$.
\subsection{The $\cS$-operator}\label{sec:A-S}
In this section we set the notations of $\cS$-operator and $S$-matrice under different basis. 

The $\cS$-operator is defined as follows.
For any cohomology classes $a,b\in H_\TT^*(\cF(m,r);\CC)$, %\xxx{Need to define $\llangle \cdots \rrangle$.}
$$
(a,\cS(b))=\llangle a,\frac{b}{z-\psi}\rrangle^{\cF(m,r),\TT}_{0,2}.
$$
The $T$-equivariant $J$-function is characterized by
$$
(J,a) = (1,\cS(a))
$$
for any $a\in H_T^*(\cF(m,r))$.

We consider several different (flat) bases for $H_\TT^*(\cF(m,r));\CC)$:
\begin{enumerate}
\item The classical canonical basis: $\phi_\alpha:=\phi_{\alpha}(q)_{q\to0}$
\item The basis dual to the classical canonical basis with respect to the $T$-equivariant Poincare pairing:
$\phi^\alpha =\Delta^{\alpha}(0)\phi_\alpha$
\item The normalized canonical basis $\hat{\phi}_\alpha=\Delta^{\alpha}(0)^{\frac{1}{2}}\phi_\alpha$.
\item The canonical basis $\phi_{\underline{\alpha}}:=\phi_{\alpha}(q)$.
\item The normalized canonical basis $\phi_{\underline{\alpha}}:=\Delta^{\alpha}(q)^{\frac{1}{2}}\phi_{\alpha}(q)$.
\end{enumerate}

For $\alpha,\beta\in \{0,1,\dots,r+m-1\}$, define:
\begin{align}
  S^{\ualpha}_{\ubeta}(z)&= (\phi_\alpha(q), \cS(\phi_\beta(q)))\\
  S^{\hat\ualpha}_{\ubeta}(z) &= (\hat\phi_\alpha(q), \cS(\phi_\beta(q)))\\
  S^{\ualpha}_{\hat\ubeta}(z) &= (\phi_\alpha(q), \cS(\hat\phi_\beta(q)))\\
  S^{\hat\ualpha}_{\hat\ubeta}(z) &= (\hat\phi_\alpha(q), \cS(\hat\phi_\beta(q)))
\end{align}

\section{Graph Sum Formula and the $R$-matrix}

We first demonstrate graph sum formulas for A-model and B-model. These subsections follow \cite{FLZ17}.

\subsection{Graph sum formula}
\

Given a connected graph $\Gamma$, we introduce the following notations:
\begin{itemize}
\item Let $V(\Gamma)$ denote the set of vertices in $\Gamma$.
\item Let $E(\Gamma)$ denote the set of edges in $\Gamma$.
\item Let $H(\Gamma)$ denote the set of half edges in $\Gamma$.
\item Let $L^0(\Gamma)$ denote the set of ordinary leaves in $\Gamma$.
\item Let $L^1(\Gamma)$ denote the set of dilation leaves in $\Gamma$.
\end{itemize}
By a half edge we mean either a leaf or an edge, together with a choice of one of the two vertices that it is attached to.
Note that the order of two vertices attached to an edge does not affect the graph sum formula in this paper.
With the above notations, we introduce the following labels:
\begin{itemize}
  \item Genus $g: V(\Gamma)\to \mathbb{Z}_{\geq 0}$;
  \item Marking $\beta: V(\Gamma)\to \{1,\cdots, m+r \}$;
  \item Height $k: H(\Gamma)\to \mathbb{Z}_{\geq 0}$.
\end{itemize}
Note that the marking on $V(\Gamma)$ induces a marking on $L(\Gamma)=L^o(\Gamma)\cup L^1(\Gamma)$ by $\beta(\ell)=\beta(v)$ where $\ell$ is attached to $v$.
Let $H(v)$ be the set of all half edges attached to $v$. Define the valence of $v\in V(\Gamma)$ as $\mathrm{val}(v)=|H(v)|$.
We say a labelled graph $\vec{\Gamma}=(\Gamma,g,\beta,r)$ is stable if $2g(v)-2+\mathrm{val}(v)> 0,\forall v\in V(\Gamma)$.
For a labelled graph $\vec{\Gamma}$, we define the genus by
$g(\vec{\Gamma})=\sum_{v\in V(\Gamma)} g(v)+ (|E(\Gamma)| - |V(\Gamma)| +1)$.
Define $\mathbf{\Gamma}_{g,N}(\cF(m,r))=\{\vec{\Gamma} \; \mathrm{ stable}: g(\vec{\Gamma})=g, |L^0(\Gamma)|=N \}$, and $\mathbf{\Gamma}(\cF(m,r))=\displaystyle\bigcup$$_{_{g,N}}\mathbf{\Gamma}_{g,N}(\cF(m,r))$.
Define the set of graphs $\tilde{\mathbf{\Gamma}}_{g,N}(\cF(m,r))$ in the same manner as $\mathbf{\Gamma}_{g,N}(\cF(m,r))$, except that the $N$ ordinary leaves are ordered.
We define weights for all leaves, edges and vertices and define the weight of a labeled graph $\vec{\Gamma}\in \mathbf{\Gamma}(\cF(m,r))$ to be the product of weights on all leaves, edges and vertices.
A graph sum formula expresses a quantity as sums of weights over all graphs.

\subsection{Givental's formula and the A-model graph sum}\label{A-model graph sum}

For a semisimple Frobenius algebra $V$, let $\{\phi_\alpha\}$ be its canonical basis. Under the identification of $V$ and $T_pV$, $\phi_\alpha$ corresponds to a tangent vector in $T_pV$.
Let $\{u_\alpha\}$ be Givental's canonical coordinates corresponding to $\phi_\alpha$, i.e., $\phi_\alpha=\frac{\partial}{\partial u^\alpha}$.
Let $U=\mathrm{diag} (u_1,\cdots,u_N)$.
Take $\Psi$ to be the base change of $\hat{\phi}_\alpha$ to $T_\alpha$, i.e., $\hat{\phi}_\alpha=\sum_\beta T_\beta\cdot\Psi_\alpha^\beta$.
By Givental's theorem\cite{Given1}, there exists a unitary $R(z)$ (i.e., $R(z)R^T(-z)=\mathrm{id}$) such that $S=\Psi R(z)e^{\frac{U}z}$ is a fundamental solution of the QDE, with $R(z)=\mathrm{id} +R_1 z+\cdots$ a formal power series in $z$.
Furthermore, $R(z)$ is unique up to a right multiplication of $\exp(a_1 z+a_3z^3+a_5z^5+\cdots)$, where $a_i$ are complex diagonal matrices.

The $\mathcal{S}$ operator is given by $(a,\mathcal{S}(b))=\left< \! \left< a, \frac{b}{z-\psi}\right> \! \right>_{0,2}^{\cF(m,r),\mathbb{T}}$. The quantization of the $\mathcal{S}$ operator relates the ancestor potential and the descendent potential via Givental's formula\cite{Given}, i.e.,
\[
D^{\cF(m,r),\mathbb{T}} (\ub) = \exp \big(F_1^{\cF(m,r),\mathbb{T}} \big) \hat{\mathcal{S}}^{-1}A^{\cF(m,r),\mathbb{T}}(\ub,\mathbf{t}).
\]
We now describe graph sum formulas for the ancestor potential $A^{\cF(m,r),T}(\ub,\mathbf{t})$ and the descendent potential with arbitrary primary insertions $F_{g,N}^{\cF(m,r),T}(\ub,\mathbf{t})$.

Let $\mathbf{u}=\mathbf{u^\alpha}T_\alpha$.
We assign weights to leaves, edges, and vertices of a labeled graph $\vec{\Gamma}\in \mathbf{\Gamma}(\cF(m,r))$ as
follows.
\begin{enumerate}
\item Ordinary leaves. To each $\ell\in L^o(\Gamma)$  we assign
\[(\mathcal{L}^{\ub})_k^\beta(\ell) = [z^k]( \sum_{\alpha=0}^{m+r-1}\frac{\ub^\alpha(z)}{\sqrt{\Delta^\alpha (q)}}
R_\alpha^\beta (-z)).\]
\item Dilaton leaves. To each $\ell\in L^1(\Gamma)$ we assign
\[(\mathcal{L}^1)_k^\beta(\ell) = [z^{k-1}]( - \sum_{\alpha=0}^{m+r-1}\frac{1}{\sqrt{\Delta^\alpha (q)}}
R_\alpha^\beta (-z)).\]
\item Edges. To an edge connecting vertices marked by $\alpha$ and $\beta$, with heights $k$ and $\ell$ at the corresponding half-edges,
we assign
\[\mathcal{E}_{k,\ell}^{\alpha,\beta}(e)=[z^kw^\ell](\frac1{z+w}(\delta_{\alpha,\beta}-\sum_{\gamma=0}^{m+r-1}
R_\gamma^\alpha(-z)R_\gamma^\beta(-w))).\]
\item Vertices. To a vertex with genus $g$, marking $\beta$ and half-edges of heights $k_1\cdots k_N$, we assign
    \[\mathcal{V}^\beta_g(v) = (\sqrt{\Delta^\beta(q)})^{2g-2+N} \left<\prod_{j=1}^N\tau_{k_j}\right>_g.\]
\end{enumerate}
Hence the weight of $\vec{\Gamma}\in \mathbf{\Gamma}(\cF(m,r))$ is:
\[
w(\vec{\Gamma}) = \dprod_{v\in V (\Gamma)} \mathcal{V}^{\beta(v)}_{g(v)}(v)
\dprod_{e\in E(\Gamma)} \mathcal{E}_{k(h_1(e)),k(h_2(e))}^{\beta(v_1(e)),\beta(v_2(e))} (e) \cdot \dprod_{\ell\in L^0(\Gamma)} (\mathcal{L}^\ub)_{k(\ell)}^{\beta(\ell)}(\ell) \dprod_{\ell\in L^1(\Gamma)} (\mathcal{L}^1)_{k(\ell)}^{\beta(\ell)}(\ell).
\]
Then it holds
\[\begin{split}
\log(A^{\cF(m,r),T}(\ub,\mathbf{t}))&= \dsum_{\vec{\Gamma}\in \mathbf{\Gamma}(\cF(m,r))} \dfrac{\hbar^{g(\vec{\Gamma})-1}w(\vec{\Gamma})}{\mathrm{Aut}(\vec{\Gamma})}= \dsum_{g\geq 0}\hbar^{g-1} \dsum_{N\geq 0}\dsum_{\vec{\Gamma}\in \mathbf{\Gamma}_{g,N}(\cF(m,r))} \dfrac{w(\vec{\Gamma})}{\mathrm{Aut}(\vec{\Gamma})}.
\end{split}
\]

We define a new weight if we have $N$ ordered variables $(\ub_1,\cdots,\ub_N)$ and $N$ ordered ordinary leaves $\{\ell_1,\cdots,\ell_N\}$.
Let
\[
\begin{split}
&S_{\underline{\hat{\alpha}}}^{\hat{\underline{\gamma}}}(z)=(\hat{\phi}_\gamma(q), \mathcal{S}(\hat{\phi}_\alpha(q))),\ \ub_j=\sum_{a\geq 0}(u_j)_az^a=\sum_\alpha \bold{u_j^\alpha} T_\alpha,\\
&(\stackrel{\circ}{\mathcal{L}^{\ub_j}})_k^\beta(\ell_j) = [z^k]( \sum_{\alpha,\gamma=0}^{m+r-1}\frac{\ub_j^\alpha(z)}{\sqrt{\Delta^\alpha (q)}} S_{\underline{\hat{\alpha}}}^{\hat{\underline{\gamma}}}(z)
R(-z)_\gamma^\beta).
\end{split}
\]
Let $\stackrel{\circ}{w}({\vec{\Gamma}})$ be the corresponding weight, then it holds similarly
\[\begin{split}
\dsum_{g\geq 0}\hbar^{g-1} \dsum_{N\geq 0} F_{g,N}^{\cF(m,r),T}(\bu_1,\cdots,\bu_N,\bold{t})= \dsum_{g\geq 0}\hbar^{g-1} \dsum_{N\geq 0} \dsum_{\vec{\Gamma}\in \mathbf{\Gamma}_{g,N}(\cF(m,r))} \dfrac{\stackrel{\circ}{w}(\vec{\Gamma})}{|\mathrm{Aut}(\vec{\Gamma})|}.
\end{split}
\]

\subsection{Eynard-Orantin recursion and the B-model graph sum}

\

Let $\omega_{g,N}$ be defined recursively by the Eynard-Orantin topological recursion
\[
\omega_{0,1}=0,\quad \omega_{0,2}=B(Y_1,Y_2)=\dfrac{dY_1\otimes dY_2}{(Y_1-Y_2)^2}.
\]
When $2g-2+N>0$, we have
\[\begin{split}
\omega_{g,N}(Y_1,\cdots,Y_N) = & \dsum_{\alpha=1}^{m+r} \mathrm{Res}_{Y\to p^\alpha} \dfrac{-\int_{\xi=Y}^{\hat{Y}} B(Y_n,\xi)}{2(\log (Y)-\log (\hat{Y}))dW_{\bT}} \big( \omega_{g-1,N-1}(Y,\hat{Y},Y_1,\cdots,Y_{N-1}) \\
& + \dsum_{g_{{}_1}+g_{{}_2}=g} \dsum_{\stackrel{I\cup J=\{1,\cdots,N-1 \}}{I\cap J\neq \emptyset}} \omega_{g_{{}_1},|I|+1}(Y,Y_I) \cdot\omega_{g_{{}_2},|J|+1}(\hat{Y},Y_J) \big),
\end{split}
\]
where $Y\neq P_\alpha$ and $\hat{Y}\neq Y$ are in a small neighborhood of $P_\alpha$ such that $W_\TT(\hat{Y}) \neq W_\TT(Y)$.

By definition, it holds that $x=W_\TT(e^y)$.
Near any critical point $v^\alpha(=\log{P_{\alpha}})$, we define $\zeta_\alpha,h_k^\alpha$ to satisfy
$x=u^\alpha-\zeta_\alpha^2, \ y=v^\alpha-\sum_{k=1}^\infty h_k^\alpha \zeta_\alpha^k$.
Expand $B(P_\alpha,P_\beta)$ in terms of $\zeta_i$ as
\[B(P_\alpha,P_\beta)=\big( \frac{\delta_{\alpha,\beta}}{(\zeta_\alpha -\zeta_\beta)^2}+ \sum_{k,\ell\geq 0} B_{k,\ell}^{\alpha,\beta}\zeta_\alpha^r\zeta_\beta^\ell\big) d\zeta_\alpha\otimes d\zeta_\beta.\]

Let
\[
\begin{split}
&\check{B}_{k,\ell}^{\alpha,\beta}=\frac{(2k-1)!!(2\ell-1)!!}{2^{k+\ell+1}} B_{k,\ell}^{\alpha,\beta},\quad
\check{h}_k^\alpha =\frac{(2k-1)!!}{2^{k-1}}h_{2k-1}^\alpha,\\
&d\xi_k^\alpha=(2k-1)!!2^{-d}Res_{P'\mapsto P_\alpha}B(P,P')(\sqrt{-1}\zeta_\alpha)^{-2d-1}.
\end{split}
\]

The B-model invariants $\omega_{g,N}$ can be expressed as graph sums. Given a labelled graph $\vec{\Gamma}\in \mathbf{\tilde \Gamma}_{g,N}(\cF(m,r))$ with $L^o(\Gamma)=\{\ell_1,\cdots,\ell_N\}$. Define the vertex factors to be $$\tilde{\mathcal{V}}^{\alpha(v)}_{g(v)}(v)=\big(\dfrac{h_1^{\alpha(v)}}{\sqrt{2}} \big)^{2-2g-\mathrm{val}(v)} \big<\dprod_{h\in H(v)} \tau_{k(h)}\big>_{g(v)}.$$.We define its weight to be
\[
\begin{split}
\tilde{w}(\vec{\Gamma}) = & (-1)^{g(\vec{\Gamma})-1+N} \dprod_{v\in V (\Gamma)} \tilde{\mathcal{V}}^{\alpha(v)}_{g(v)}(v)
\dprod_{e\in E(\Gamma)} \check{B}_{k(e),\ell(e)}^{\alpha(v_1(e)),\alpha(v_2(e))} (e) \\ & \cdot \dprod_{j=1}^N \dfrac1{\sqrt{-2}}d\xi_{k(\ell_j)}^{\alpha(\ell_j)}(Y_j) \dprod_{\ell\in L^1(\Gamma)} \big( \dfrac1{\sqrt{-2}} \big)\check{h}_{k(\ell)}^{\alpha(\ell)}.
\end{split}
\]

We cite here the Theorem 3.7 in \cite{DOSS}.
\begin{thms}
For $2g-2+N>0$, it holds
\[
\omega_{g,N}=\dsum_{\vec{\Gamma}\in \mathbf{\Gamma}_{g,N}(\cF(m,r))} \dfrac{\tilde{w}(\vec{\Gamma})}{|\mathrm{Aut}(\vec{\Gamma})|}.
\]
\end{thms}

\subsection{A-model large radius limit}

\

In 3.4 and 3.5, we assume $p$ and $z$ to be negative real numbers.

We denote  $Q_1$ as the chart $\cF(m,r)\setminus\{[0:1]\}$, and $Q_2$ as the chart $\cF(m,r)\setminus\{[1:0]\}$.
By Tseng \cite{Tseng10} (see also Zong \cite{Zong}), we have
\[
\left.\big( R_j^i\big)\right|_{t=0,q=0} = \mathrm{diag} \big( (P_\sigma)_j^i\big), \;\mathrm{on} \; Q_\sigma,  \;\mathrm{ for }\; \sigma=1,2;
\]
\[
(P_\sigma)_j^i =\dfrac1{|G_\sigma|}\dsum_{(h)}\chi_{\alpha_j}(h)\chi_{\alpha_i}(h^{-1})\cdot \exp\left[ \dsum_{t=1}^\infty \dfrac{(-1)^t}{t(t+1)}B_{t+1}(c_\sigma(h)) \big( \dfrac{z}{w_\sigma}\big)^t \right].
\]
Let $\sigma=1$. Then we have the following:
\begin{itemize}
  \item $G=G_1=\mathbb{Z}/r\mathbb{Z}$.
  \item $V_{\alpha_{1+j}}=\mathbb{C}$, with $\mathbb{Z}/r\mathbb{Z}$ action: $\bar{t} \circ z=e^{2\pi i \frac{tj}{r}}\cdot z$. Then $\chi_{\alpha_{1+j}}(\bar{t})=e^{2\pi i \frac{tj}{r}}$.
  \item $T=\{ (\lambda_1,\lambda_2)\}$ acts on $Q_1$ by $(\lambda_1,\lambda_2)\circ z = ( \lambda_2^{\frac{m}{r}}\lambda_1)z$, i.e., $w_1= \lambda_2^{\frac{m}{r}}\lambda_1$.
  \item $c_\sigma( e^{2\pi i\frac{t}{r}})=\frac{t}{r}$, where $0\leq t <r$.
\end{itemize}
%With the above preparations, we have
%\[
%(P_1)_j^k=\dfrac1m \dsum_{h=0}^{m-1} e^{2\pi i \frac{h(j-k)}{m}} \cdot \exp \left[ \dsum_{t=1}^{\infty} \dfrac{(-1)^t}{t(t+1)} B_{t+1}\big(\dfrac{h}m\big) %\cdot \big( \dfrac{z}{\lambda_1\lambda_2^{n/m}}\big)^t  \right].
%\]

By \cite{Kacz11} we have
$\log \Gamma(z+s)=(z+s-\frac12) \log z -z +\frac12 \log2\pi +\sum_{t=1}^\infty \frac{(-1)^t \cdot B_{t+1}(s)}{t(t+1)}\frac1{z^t}$.
Let $\lambda=\frac{\lambda_1\lambda_2^{m/r}}{z}$, $\omega_{\alpha,\beta}=e^{\frac{2\pi i(\alpha-\beta)}{r}}$. Then we have
\[
(P_1)_\beta^\alpha = \dfrac{e^\lambda}{r\sqrt{2\pi\lambda}}\dsum_{h=0}^{r-1} (\omega_{\alpha,\beta})^{-h}\Gamma\big(\lambda+\dfrac{h}{r}\big) \lambda^{1-\lambda-\frac{h}{r}}.
\]

\subsection{B-model large radius limit}

\

Next we calculate the B-model $R$-matrix $\tilde{R}_\alpha^\beta$ while $\tilde{q}\to 0$.
Let
\[\tilde{R}_\alpha^\beta(\tilde{q})=\frac{\sqrt{-z}}{2\sqrt{\pi}}\dint_{\gamma_{\beta}}e^{\frac{W_\TT(P_{\beta})-W_\TT(Y)}z}d\xi_{\alpha,0},
\]
where:
\begin{itemize}
  \item $P_\beta$ are critical points of $W_\TT$ for $\beta=0,\cdots,m+r-1$, and in $q\to0$ case we assume $\beta=0,\cdots,r-1$;
  \item for a fixed $\beta$, $\gamma_{\beta}=W_\TT^{-1}(W_\TT(P_\beta)+[0,+\infty))$;
  \item $\xi_{\alpha,0}=\frac{1}{\sqrt{-1}}\sqrt{\frac{2}{\Delta^{\alpha}(q)}}\frac{P_{\alpha}}{Y-P_{\alpha}}$.
\end{itemize}
Noticing that $\tilde{R}_\alpha^\beta(\tilde{q})$ only involves terms of differences, it remains unchanged if we add a constant to $W_\TT$. More specifically, we replace $W_\TT$ by $W_\TT-(\sum_{\ell=1}^{r-1} w_\ell\log \tilde{q}_\ell + \sum_{\ell=1}^{m} w_{-\ell}\log \tilde{q}_{-\ell})$.

Further let $(\tilde{R}_{1})_{\alpha}^{\beta}$ be the submatrix of the first $r$ columns and $r$ rows of $\lim_{\tilde{q}\to 0}\tilde{R}_\alpha^\beta(\tilde{q})$. It may be computed from $\int_{\gamma_{\beta}}\exp(\frac{\widetilde{W_\TT}(Y)-\widetilde{W_\TT}(p^\beta)}z) \theta_\alpha$ with $\widetilde{W_\TT}=Y^{r}+p\log Y$.
In this new set-up, we have the following results:
\begin{itemize}
  \item $P_\beta$s are roots of $0=\frac{\partial \widetilde{W_\TT}}{\partial \log Y}=r( Y^{r}-\dfrac{p}{r})$. This gives $P_\beta=\sqrt[n]{\frac{p}{r}}\cdot e^{2\pi i \frac\beta{r}}$.
  \item We claim that $\gamma_{\beta}=(0,+\infty)\cdot P_\beta$.
  In fact, direct calculation shows that
  $\widetilde{W_\TT}(t\cdot P_\beta)-\widetilde{W_\TT}(P_\beta)
  =\frac{p}{r}(t^{r}-1-r\log t)\geq 0$.

  \item
  With $\Delta^{\alpha}(0)=rp$.
  ,we get $d\xi_{\alpha,0}=\frac{-P_\alpha}{(Y-P_{\alpha})^{2}}\cdot \sqrt{\frac{2}{rp}}dY$.
\end{itemize}

Let $\mu=\frac{p}{rz},\omega_{\alpha,\beta}=e^{\frac{2\pi i(\alpha-\beta)}{r}},s=\mu t^{r}$.
Integrating by parts, the equation equals:
\[\tilde{R}_{\alpha}^{\beta}(\tilde{q})_{\tilde{q}\to0}=\frac{e^{\mu}}{r\sqrt{2\pi\mu}}\int_{s=0}^{+\infty} e^{-s} \big(\frac{s}\mu \big)^{\mu} d\frac{1}{ \big(\frac{s}\mu \big)^{\frac1m}-\omega_{\alpha,\beta}}
=  \frac{e^{\mu}}{r\sqrt{2\pi\mu}}  \dsum_{h=0}^{r-1} \Gamma\big(\mu+\dfrac{h}{r} \big) \cdot \mu^{1-\mu-\frac{h}{r}}\cdot (\omega_{\alpha,\beta})^{-h}.\]
This directly gives the result.

\begin{prop}
It holds that
$(\tilde{R}_1)_\alpha^\beta = (P_1)_\beta^\alpha, \; \mathrm{for }\; 0\leq \alpha,\beta \leq r-1$,
if we identify $\lambda$ and $\mu$.
\end{prop}

\subsection{The general case}

\

It is obvious that $\tilde{R}_\alpha^\beta(q)|_{q=0}=0=\tilde{R}_\beta^\alpha (q)|_{q=0}$ for $1\leq \alpha\leq m < m+1\leq \beta\leq m+r$.
Since both $\Psi Re^{\frac{U}z}$ and  $\Psi \tilde{R}e^{\frac{U}z}$ are solutions to the QDE on the Frobenius algebra $QH^\ast (\cF(m,r))\cong \mathrm{Jac}(\cF(m,r))$, we have by Givental's theorem$\tilde{R}(\tilde{q})=R(q)\cdot A$, where $A=\exp(a_1 z+a_3z^3+a_5z^5+\cdots)$, with $a_i$'s diagonal matrices, and their diagonal entries are scalar.
Considering the submatrix consisting of the first $n$ columns and the first $n$ rows of $R$ and $\tilde{R}$, by the previous proposition we have $\lim_{q\to 0}P_1(q)=\tilde{R}=\lim_{q\to 0}P_1(q)\big|_{r\times r}$.
Comparing the diagonal elements, we find the submatrix of first $n$ columns and $n$ rows of $A$ as $A|_{r\times r}=I_{r\times r}$.
Similarly, moving to the other chart we have $A|_{m\times m}=I_{m\times m}$. Hence $A=I$. This gives the following proposition.
\begin{prop}\label{A-B R-matrix equivalence}
It holds that
\[R(q)=\tilde{R}(\tilde{q}).\]
\end{prop}

\section{All genus equivariant mirror symmetry}

\subsection{Calculations of the graph sum formula}
\

First observe that
$\sqrt{\Delta^\alpha} h_1^\alpha=\sqrt{2}$.

Let
\[
\begin{split}
&\xi_{\alpha,0}=\frac1{\sqrt{-1}}\sqrt{\frac2{\Delta^\alpha}} \frac{P_\alpha}{Y-P_\alpha},
\theta=\frac{d}{dW_{\bT}},
W_k^\alpha =d((-1)^k\theta^k(\xi_{\alpha,0})),\\
&(\tilde{u}_j)_k^\alpha=[z^k]\sum_\beta S^{\hat{\underline{\alpha}}}_{\hat{\underline{\beta}}}(z)\frac{u_j^\beta(z)}{\sqrt{\Delta^\beta(q)}}.
\end{split}
\]
Note that $d\xi_{\alpha,0}=d\xi_0^\alpha$. The last equation is also equivalent to 
\begin{equation*}
    (u_j)_k^\alpha=[z^k]\sqrt{\Delta^{\alpha}(q)}\sum_\beta S_{\hat{\underline{\alpha}}}^{\hat{\underline{\beta}}}(-z)\tilde{u}_j^\beta(z)
\end{equation*}
This comes from the unitary condition $R(z)R^{T}(-z)=I$. Written in the matrix form, we have:
\begin{equation}
    \bu_{j}=\bold{\Delta}^{\frac{1}{2}}(\bold{S}_{\hat\ubeta}^{\hat\ualpha})^{T}\bold{\tilde{\bu}_{j}}
\end{equation}
where $\bu_{j}=(u_{j}^{\alpha})^{T}$, and $\tilde{\bu}_{j}=(\tilde{u}_{j}^{\hat\alpha})^{T}$.
This also implies 
\begin{equation*}
\dsum_{\alpha}\bu^{\alpha}\phi_{\alpha}=(\phi_{\alpha})^{T}\bold{\Delta}(\bold{S}^{\ualpha}_{\ubeta}(-z))^{T}\bold{\Delta}^{\frac{1}{2}}\tilde{\bu}^{\hat\beta}=\dsum_{\beta}-z\frac{\partial J(-z)}{\partial u^{\beta}}\tilde{\bu}^{\beta}
\end{equation*}
\begin{thm}\label{main}
By identifying $W_k^\alpha(Y_j)$ and $\sqrt{-2}(\tilde{u}_j)_k^{\hat\alpha}$, we have
\[
\omega_{g,N}=(-1)^{g-1+N}F_{g,N}^{\cF(m,r),\TT} (\mathbf{u_1,\cdot\cdot\cdot,u_N,t})
\]
for $2g-2+n>0$ and $N>0$.
\end{thm}
\begin{proof} We prove by direct calculation as follows,

\begin{enumerate}
\item{Vertices}:
This follows from $\sqrt{\frac{\Delta^{\alpha(v)}}{2}}h_1^{\alpha(v)}=1$.

\item {Edges}:
By \cite{DOSS},$R_\beta^\alpha(z)=f_\beta^\alpha\big( -\frac1z\big)$ and the contribution of edges to weight in B-model is

\[
\check{B}_{k,\ell}^{\alpha,\beta}(e)
=[u^{-k}v^{-\ell}]\frac{uv}{u+v}\big(\delta_{\alpha\beta}-\sum_{\gamma=1}^{m+r}f_\gamma^\alpha(u)f_\gamma^\beta(v)\big)
=\mathcal{E}_{k,\ell}^{\alpha,\beta}(e).
\]

\item {Ordinary leaves}:
By $\frac1{z+w} = \frac1{z} \sum_{s\geq 0}\big( -\frac{w}z\big)^s$ we know $-\check{B}_{k-1-i,0}^{\alpha,\beta} = [z^{k-i}] R_\beta^\alpha(-z)$.
From \cite{FLZ16} we know $d\xi_k^\alpha = W_k^\alpha -\sum_{i=0}^{k-1}\sum_\beta \check{B}_{k-1-i,0}^{\alpha,\beta}W_i^\beta$.

It holds after identifying $\frac1{\sqrt{-2}}W_k^\alpha(Y_j)$ and $(\tilde{u}_j)_k^{\hat\alpha}$ that
\[
\begin{split}
\big(\mathcal{L}_d^{\bold{u_j}}\big)_{k(\ell_j)}^{\alpha(\ell_j)}(\ell_j)
& = [z^{k(\ell_j)}] \dsum_{\beta=0}^{m+r-1}\dsum_{i=0}^{k(\ell_j)} (\tilde{u}_j)_i^{\hat\beta} \cdot z^i \cdot  R_\beta^{\alpha(\ell_j)} (-z)\\
& = \dsum_{i=0}^{k(\ell_j)}\dsum_{\beta=0}^{m+r-1}(\tilde{u}_j)_i^{\hat\beta} \big( [z^{k(\ell_j)-i}]R_\beta^{\alpha(\ell_j)} (-z)\big)\\
& = \dfrac1{\sqrt{-2}}d\xi_{k(\ell_j)}^{\alpha(\ell_j)}(Y_j).
\end{split}
\]

\item {Dilaton leaves}:
By \cite{FLZ16} and the relation $R_\beta^\alpha(z)=f_\beta^\alpha \big( \frac{-1}z \big)$, we have
\[\check{h}_{k(\ell)}^\alpha = [u^{1-k(\ell)}] \sum_{\beta=1}^{m+r} \sqrt{-1}h_1^\beta R_\beta^{\alpha(\ell)}\big(\frac{-1}{u}\big)
= [z^{k(\ell)-1}] \sum_{\beta=1}^{m+r} \sqrt{-1}h_1^\beta R_\beta^{\alpha(\ell)}\big(-z\big).\]

By $h_1^\beta=\sqrt{\frac2{\Delta^\alpha}}$, we know that
$\big( \mathcal{L}^1\big)_{k(\ell)}^{\alpha(\ell)}(\ell)=\big( -\frac1{\sqrt{-2}}\big)\check{h}_{k(\ell)}^{\alpha(\ell)}$.
\end{enumerate}
\end{proof}

\subsection{The Laplace Transform}
\

Following Iritani \cite{Iritani} with slight modification, we define as follows \cite{Fang20}. The result in this part is a generalization of \cite{Tang}.
\begin{Def}[equivariant Chern character] We define equivariant Chern character
\[\widetilde{ch}_z: K_\TT(\cF(m,r))\rightarrow H_{CR,\TT}^\ast(\cF(m,r),\mathbb{Q})\left[\left[ \dfrac{p}{z} \right]\right]\] by the following two properties which uniquely characterize it:
\begin{enumerate}
  \item $\widetilde{ch}_z(\varepsilon_1\oplus\varepsilon_2)=\widetilde{ch}_z(\varepsilon_1)+\widetilde{ch}_z(\varepsilon_2)$.
  \item Let $\cX=\cF(m,r)$, $I\cX=\bigsqcup_{v\in\mathrm{Box}(\Sigma)}\cX_{v}$.If $\mathcal{L}$ is a $T$-equivariant line bundle on $\cF(m,r)$, then
  $$\widetilde{ch}_z(\mathcal{L})=\bigoplus_{v\in\mathrm{Box}(\Sigma)}\exp\big( 2\pi i(-\frac{(c_1)_\TT(\mathcal{L}_{v})}{z}+\age_{v}(\cL))\big)$$
  where $\cL_{v}=\cL|_{\cX_{v}}$, and $\age_{v}(\cL)$ is the age of $\cL_{v}$ along $\cX_{v}$.
  \item $\widetilde{\Gamma}_{z}(\cL)=\bigoplus_{v\in\mathrm{Box}(\Sigma)}(-z)^{1+\frac{(c_1)_\TT(\cL_{v})}{z}-\age_{v}(\cL)}\Gamma(1+\frac{(c_1)_\TT(\cL_{v})}{z}-\age_{v}(\cL))$
\end{enumerate}
\end{Def}

\begin{Def}[equivariant $K$-theoretic framing]
For $\forall \varepsilon\in K_\TT(\cF(m,r))$, we define the $K$-theoretic framing of $\varepsilon$ by
$\kappa(\varepsilon)= \widetilde\Gamma_{z}\big(T\cX\big) \widetilde{ch}_z(\varepsilon)$.
\end{Def}

\begin{Def}[equivariant SYZ $T$-dual] Let $\mathcal{L}=\mathcal{O}_{\cF(m,r)}(\ell_1 p_1+\ell_2 p_2)$ be an equivariant ample line bundle on $\cF(m,r)$, where $\ell_1,\ell_2 \in \mathbb{Z}$, such that $\ell_1+\ell_2 >0$. We define the equivariant SYZ $T$-dual $\mathrm{SYZ}(\mathcal{L})$ of $\mathcal{L}$ be the figure below:
\end{Def}
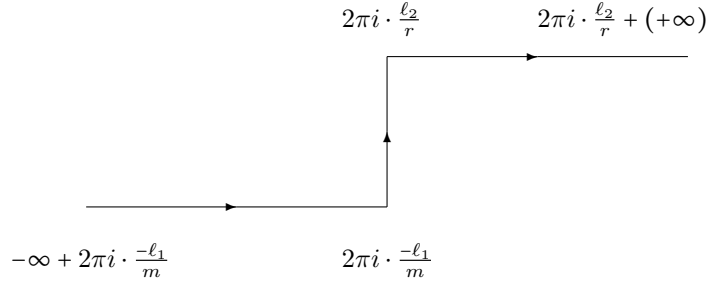
\begin{figure}[h]
\begin{center}
\setlength{\unitlength}{2mm}
\begin{picture}(60,20)
%\linethickness{1pt}

\put(10,6){\vector(1,0){10}}
\put(20,6){\line(1,0){10}}
\put(30,6){\vector(0,1){5}}
\put(30,11){\line(0,1){5}}
\put(30,16){\vector(1,0){10}}
\put(40,16){\line(1,0){10}}

\put(5,2){$-\infty+2\pi i \cdot \frac{-\ell_1}{m}$}
\put(27,2){$2\pi i\cdot\frac{-\ell_1}{m}$}
\put(27,18){$2\pi i\cdot\frac{\ell_2}{r}$}
\put(40,18){$2\pi i \cdot\frac{\ell_2}{r}+(+\infty)$}

\end{picture}
   %\caption{}
  \label{fig1}
\end{center}
\end{figure}

Extend the definition additively to the equivariant $K$-theory group $K_\TT(\cF(m,r))$.
By \cite{Fang20}:
\begin{thms}
\[
\left<\!\left< \dfrac{\kappa(\mathcal{L})}{z(z-\psi)}\right>\!\right>^{\cF(m,r),\TT}_{0,1}=\dint_{\mathrm{SYZ}(\mathcal{L})} e^{\frac{W_\TT}z} dy.
\]
\end{thms}

\begin{cor}\label{T-dual transformed}
\

\begin{enumerate}
  \item By string equation: \[
  \dint_{\mathrm{SYZ}(\mathcal{L})} e^{\frac{W_\TT}z} dy =\left<\!\left< \dfrac{\kappa(\mathcal{L})}{z(z-\psi)}\right>\!\right>^{\cF(m,r),\TT}_{0,1}
  =\left<\!\left< 1, \dfrac{\kappa(\mathcal{L})}{z-\psi}\right>\!\right>^{\cF(m,r),\TT}_{0,2};
  \]
  \item Integrating by parts, \[ -\left<\!\left< \dfrac{\kappa(\mathcal{L})}{z-\psi}\right>\!\right>^{\cF(m,r),\TT}_{0,1} =-z \dint_{\mathrm{SYZ}(\mathcal{L})} e^{\frac{W_\TT}z} dy = \dint_{\mathrm{SYZ}(\mathcal{L})} e^{\frac{W_\TT}z} ydx.
  \]
\end{enumerate}
\end{cor}

Define
\[
S_{\underline{\beta}}^{\underline{\hat{\alpha}}}(z)=\left<\!\left< \phi_\beta(q),\frac{\hat{\phi}_\alpha(q)}{z-\psi}\right>\!\right>_{0,2}^{\cF(m,r),\TT},\\
S_{\widehat{\underline{\beta}}}^{\kappa(\mathcal{L)}} (z) =\left<\!\left< \hat{\phi}_\beta(q),\frac{\kappa(\mathcal{L)}}{z-\psi} \right>\!\right>_{0,2}^{\cF(m,r),\TT}.\]
More generally, we have
\begin{prop}
\[
S_\beta^{\widehat{\underline{\alpha}}}(z)=  -z \dint_{y\in \gamma_\beta(\mathcal{L})} e^{\frac{W_\TT}z} \dfrac{d\xi_{\alpha,0}}{\sqrt{-2}}.
\]
\[S_{\hat{\underline{\beta}}}^{\kappa(\mathcal{L})}(z)=-z\dint_{y\in{\mathrm{SYZ}(\mathcal{L})}} e^{\frac{W_\TT}z}\dfrac{d\xi_{\beta,0}}{\sqrt{-2}}.\]
\end{prop}

\begin{proof}

We prove the second equation as an example. The first one may be proved in a similar way.

Let $f(Y)=\frac{\partial W_{\bT}}{\partial y}$. Then $\Delta^\alpha=P_\alpha\cdot f'(P_\alpha)$.
By $\hat\phi_\alpha =\sqrt{\Delta^\alpha}\phi_\alpha$ and $\xi_{\alpha,0}=\frac1{\sqrt{-1}} \sqrt{\frac2{\Delta^\alpha}} \frac{P_\alpha}{Y-P_\alpha}$, the desired proposition is equivalent to
\[
\left<\!\left< \phi_\beta,\dfrac{\kappa(\mathcal{L)}}{z-\psi} \right>\!\right>_{0,2}^{\cF(m,r),\TT}  = z\dint_{y\in{\mathrm{SYZ}(\mathcal{L})}} e^{\frac{W_\TT}z} d\dfrac{P_\beta}{(Y-P_\beta)\Delta^\beta}.
\]

By \ref{T-dual transformed}, we have
\[\left<\!\left< 1,\frac{\kappa(\mathcal{L)}}{z-\psi} \right>\!\right>_{0,2}^{\cF(m,r),\TT}  = \int_{y\in{\mathrm{SYZ}(\mathcal{L})}} e^{\frac{W_\TT}z} \frac{dY}Y.\]
Applying $z\dfrac{\partial}{\partial q_i}$ to both sides,  we have
\[\left<\!\left< X^i,\frac{\kappa(\mathcal{L)}}{z-\psi}\right>\!\right>_{0,2}^{\cF(m,r),\TT} = \int_{y\in\mathrm{SYZ}(\mathcal{L})} Y^i e^{\frac{W_\TT}z} \frac{dY}Y.\]
Note that for the left hand side derivation, we have opened the double bracket and used the string equation.

This implies
\[\begin{split}
& \left<\!\left< \phi_\beta,\dfrac{\kappa(\mathcal{L)}}{z-\psi} \right>\!\right>_{0,2}^{\cF(m,r),\TT} \\  = & \dint_{y\in{\mathrm{SYZ}(\mathcal{L})}} e^{\frac{W_\TT}z} \dfrac{f(Y)Y^m}{Y-P_\beta} \dfrac{dY}Y \cdot \dfrac1{(Y^mf(Y))'|_{P_\beta}}\\
= & -z \dint_{y\in{\mathrm{SYZ}(\mathcal{L})}}\dfrac{1}{(Y-p^\beta)f'(p^\beta)} d e^{\frac{W_\TT}z} - \dint_{y\in{\mathrm{SYZ}(\mathcal{L})}} e^{\frac{W_\TT}z} \dfrac{f(Y)(Y^m-(P_\beta)^m)}{(Y-P_\beta)(P_\beta)^mf'(P_\beta)} \dfrac{dY}{Y} \\
= & z \dint_{y\in{\mathrm{SYZ}(\mathcal{L})}} e^{\frac{W_{\bT}}z} d\dfrac{1}{(Y-P_\beta)f'(P_\beta)} - \dint_{y\in{\mathrm{SYZ}(\mathcal{L})}} e^{\frac{W_\TT}z} \dfrac{f(Y)(Y^m-(P_\beta)^m)}{(Y-P^\beta)(P_\beta)^mf'(P_\beta)} \dfrac{dY}{Y}\\
= & z \dint_{y\in \mathrm{SYZ}(\mathcal{L})} e^{\frac{W_\TT}z} d\dfrac{P_\beta}{(Y-P_\beta)\Delta^\beta}.
\end{split}\]
The last equation holds because
\[\int_{y\in{\mathrm{SYZ}(\mathcal{L})}} e^{\frac{W_\TT}z} g(Y)f(Y)\frac{dY}Y =
\left<\!\left< g(Y)f(Y),\frac{\kappa(\mathcal{L})}{z-\psi}\right>\!\right>_{0,2}^{\cF(m,r),\TT}
=\left<\!\left< 0,\frac{\kappa(\mathcal{L})}{z-\psi}\right>\!\right>_{0,2}^{\cF(m,r),\TT}
=0,\]
where $g(Y) =\dfrac{1}{(P_\beta)^mf'(P_\beta)} \dfrac{(Y^m-(P_\beta)^m)}{(Y-P_\beta)}$ (is a polynomial of $Y$).
\end{proof}

Integrating the second equation by parts, we have
\[S_{\hat{\underline{\beta}}}^{\kappa(\mathcal{L})} (z) =-z^{k+1}\dint_{y\in \mathrm{SYZ}(\mathcal{L})}  e^{\frac{W_\TT}z} \frac{W_k^\beta}{\sqrt{-2}}.\]
Also notice that
\[\begin{split}
&\dsum_{\gamma=0}^{m+r-1} S_\alpha^{\widehat{\underline{\gamma}}} (z) S_\beta^{\widehat{\underline{\gamma}}}(-z) = (\phi_\alpha(0),\phi_\beta(0))=\frac{\delta_{\alpha\beta}}{\Delta^{\alpha}(0)},\\
&\dsum_{\alpha=0}^{m+r-1} S_\beta^{\widehat{\underline{\alpha}}} (-z) S_{\widehat{\underline{\alpha}}}^{\kappa(\mathcal{L)}}(z) = (\phi_\beta(0),\mathcal{K(L)}).
\end{split}
\]

\begin{thm} It holds that
\[
\dint_{y_1\in \mathrm{SYZ}(\mathcal{L}_1)}\cdots \dint_{y_\ell\in \mathrm{SYZ}(\mathcal{L}_\ell)} e^{\frac{W_{\bT}(y_1)}{z_1}+\cdots \frac{W_{\bT}(y_N)}{z_N}} \omega_{g,N} =
(-1)^{g-1} \left<\!\left< \dfrac{\kappa(\mathcal{L}_1)}{z_1-\psi_1}, \cdots,\dfrac{\kappa(\mathcal{L_N)}}{z_N-\psi_N}  \right>\!\right>_{g,N}.
\]
\end{thm}

\begin{proof}
By definition,
\[
\tilde{\ub}_j^{\hat{\alpha}}(z)=\dsum_{\beta=1}^{m+r}\sqrt{\Delta^\alpha(q)}\left<\!\left< \phi_\alpha(q),\dfrac{\phi_\beta(q)}{z-\psi} \right>\!\right>_{0,2}^{\cF(m,r),\TT} \ub_j^\beta(z).
\]
Taking the Laplace transform of $\omega_{g,N}$, by Theorem 1 and definition of $\bar{\ub}_i$, we get

\[\begin{split}
& \dint_{y_1\in \mathrm{SYZ}(\mathcal{L}_1)}\cdots \dint_{y_N\in \mathrm{SYZ}(\mathcal{L}_N)} e^{\frac{W_{\bT}(y_1)}{z_1}+\cdots \frac{W_{\bT}(y_N)}{z_N}} \omega_{g,N} \\
= & \dint_{y_1\in \mathrm{SYZ}(\mathcal{L}_1)}\cdots \dint_{y_N\in \mathrm{SYZ}(\mathcal{L}_N)} e^{\frac{W_{\bT}(y_1)}{z_1}+\cdots \frac{W_{\bT}(y_N)}{z_N}} (-1)^{g-1-N} \Big( \dsum_{\beta_i,\alpha_i}\left<\!\left< \dprod_{i=1}^N \tau_{\alpha_i}(\phi_{\beta_i}(0)) \right>\!\right>_{g,N} \\
& \cdot \dprod_{i=1}^N (\bar{\ub}_i)_{\alpha_i}^{\beta_i}\Big|_{(\tilde{\ub}_j)_k^{\beta}=\frac1{\sqrt{-2}}W_k^\beta(y_j)} \Big)\\
= & \dint_{y_1\in \mathrm{SYZ}(\mathcal{L}_1)}\cdots \dint_{y_N\in \mathrm{SYZ}(\mathcal{L}_N)} e^{\frac{W_{\bT}(y_1)}{z_1}+\cdots \frac{W_{\bT}(y_N)}{z_N}}  (-1)^{g-1-N} \Bigg[ \dsum_{\beta_i,\alpha_i}\left<\!\left<\dprod_{i=1}^N \tau_{\alpha_i}(\phi_{\beta_i}(0)) \right>\!\right>_{g,N} \\
& \cdot \dprod_{i=1}^N \big( \Delta^{\beta_i}\dsum_{\alpha=0}^{m+r-1}\dsum_{k\in\mathbb{Z}_{\geq 0}} [z_i^{\alpha_i-k}] S_{\beta_i}^{\widehat{\underline{\alpha}}} (-z_i) \dfrac{W_k^\alpha(y_i)}{\sqrt{-2}} \big) \Bigg]\\
= & (-1)^{g-1+N} \left[\dsum_{\beta_i,\alpha_i}\left<\!\left< \dprod_{i=1}^N\tau_{\alpha_i}(\phi_{\beta_i}(0)) \right>\!\right>_{g,N} \cdot \dprod_{i=1}^N \big( \Delta^{\beta_i}\dsum_{\alpha=0}^{m+r-1}\dsum_{k\in\mathbb{Z}_{\geq 0}} ([z_i^{\alpha_i-k}] S_{\beta_i}^{\widehat{\underline{\alpha}}} (-z_i) )S_{\widehat{\underline{\alpha}}}^{\kappa(\mathcal{L}_1)} (z_i) (-z_i^{-k-1}) \big) \right] \\
= & (-1)^{g-1} \dsum_{\beta_i,\alpha_i}\left<\!\left<\dprod_{i=1}^N \tau_{\alpha_i}(\phi_{\beta_i}(0)) \right>\!\right>_{g,N}  \dprod_{i=1}^N  \Delta^{\beta_i}(\phi_{\beta_i}(0),\kappa(\mathcal{L}_i))z_i^{-\alpha_i-1}\\
= & (-1)^{g-1} \left<\!\left< \dfrac{\kappa(\mathcal{L}_1)}{z_1-\psi_1}, \cdots,\dfrac{\kappa(\mathcal{L}_{N})}{z_N-\psi_N}  \right>\!\right>_{g,N}.
\end{split}\]
\end{proof}
\section{The large radius limit and the Bouchard-Mari\~{n}o conjecture}
\label{sec:BM}

In this section we consider the large radius limit $q_{*}\to 0$ and let $\w_i=0$ for $i\neq r$, $q_i=0$ for all $i\in I_{0}$. 
\subsection{Equivariant $J$-function under large radius limit}
In this subsection we calculate the first order differentials of equivariant $J$-function under the large radius limit. By direct calculation, we notice since $q_{*}=0$, and $\w_{-m}=0$, we have:
\begin{equation*}
    \bold{1}_{i,m}=0
\end{equation*}
for $i\leq0$. So under this limit, we have:
\begin{equation*}
     I^{\mathbb{T}}(q,z)_{ q_{*},\w_{i}\to0}=(\dsum_{d_{i}\geq0}(\dprod_{i=0}^{r-1}q_{i}^{d_{i}}){z^{-\lceil\sum_{i=0}^{r-1}\frac{(r-i)d_{i}}{r}\rceil}}\cdot\frac{\Gamma(\frac{\w_{r}}{z}+1-{\{-\tilde{d}_{r}}\})}{\dprod_{i\in I_{0}}d_{i}!\Gamma(\frac{\w_{r}}{z}+\tilde{d}_{r}+1)}\mathbf{1}_{v(d)}.
\end{equation*}
In this case, $v(d)=\dsum_{i=0}^{r-1}id_{i}$ modulo $r$, and $\tilde{d}_{r}=-\dsum_{i=0}^{r-1}\frac{id_{i}}{r}$.

Consider the flat basis $\{X^{i}\}$ for $i=0,1,\dots,n-1$, and the flat coordinate $\tilde{t}_{i}$, we have $\tilde{t}_{i}=q_{i}$. We write $ \tilde{\partial}_{i}J(z):=\big(\frac{z\partial J(z)}{\partial{\tilde{t}}_{i}}\big)$.So by direct calculation:
\begin{equation*}
    \big(\tilde{\partial}_{0}J(z)\big)_{q_{*},\w_{j}\to0}=(\dsum_{d_{i}\geq0}(\dprod_{j=0}^{r-1}q_{j}^{d_{j}}){z^{-\lceil\sum_{j=0}^{r-1}\frac{(r-j)d_{j}}{r}\rceil}}\cdot\frac{\Gamma(\frac{\w_{r}}{z}+1-{\{-\tilde{d}_{r}}\})}{\dprod_{j=0}^{r-1}d_{i}!\Gamma(\frac{\w_{r}}{z}+\tilde{d}_{r}+1)}\mathbf{1}_{v(d)}.
\end{equation*}
and for $0<i<r$
\begin{equation*}
    \big(\tilde{\partial}_{i}J(z)\big)_{q_{*},\w_{j}\to0}=(\dsum_{d_{i}\geq0}(\dprod_{j=0}^{r-1}q_{j}^{d_{j}}){z^{-\lceil-\frac{i}{r}+\sum_{j=0}^{r-1}\frac{(r-j)d_{j}}{r}\rceil}}\cdot\frac{\Gamma(\frac{\w_{r}}{z}+1-\{-\tilde{d}_{r}+\frac{i}{r}\})}{\dprod_{j=0}^{r-1}d_{j}!\Gamma(\frac{\w_{r}}{z}+\tilde{d}_{r}-\frac{i}{r}+1)}\mathbf{1}_{v(d+p_{i})}.
\end{equation*}

\subsection{B-model under large radius limit}
Under the limit $q_{*},\w_{i}\to0$, our mirror curve becomes
$$x=Y^{r}+\dsum_{i=0}^{r-1}q_{i}Y^{i}-r\w_{r}\log Y.$$
When $\w_{r}=\frac{1}{r}$, we get the curve:
\begin{equation*}
    X=Ye^{-Q(Y)}
\end{equation*}
when setting $X=e^{-x}$, and $Q(Y)=Y^{r}+\dsum_{i=0}^{r-1}q_{i}Y^{i}$.

We notice that in this limit we have $n$ critical points $P_{\alpha}$ of $W_{\TT}(Y)$, and we index them by $\alpha=0,1,\dots,r-1$.

We write $$W_{\alpha}(Y)=\frac{\partial W_{\TT}}{\partial y}\cdot\frac{1}{Y-P_{\alpha}}.$$

We have $\Delta^{\alpha}(q)=P_{\alpha}W_{\alpha}(P_{\alpha}).$

We have
$$
\xi_{\alpha,0}=\frac{1}{\sqrt{-1}}\sqrt{\frac{2}{\Delta^{\alpha}(q)}}\frac{P_{\alpha}}{Y-P_{\alpha}}.
$$

\begin{itemize}
    \item For a holomorphic function $f(Y)$ on $D_\epsilon$, let
    \[
        \fh_X(f) := \sum_{\mu\in\ZZ_{\neq 0}} \Big(\mathop{\Res}\limits_{Y\rightarrow 0}f(Y)X^{-\mu}\frac{dX}{X}\Big)X^{\mu}.
    \]
    \item For a holomorphic differential form $\theta(Y)$ on $D_\epsilon$, let 
    \[     
        \fh_X(\theta) := \sum_{\mu\in\ZZ_{\neq 0}} \Big(\mathop{\Res}\limits_{Y\rightarrow 0} \theta(Y) X^{-\mu}\Big)\frac{X^{\mu}}{\mu}.
    \]
    If $f$ is a holomorphic function on $D_\epsilon$ such that $\theta(Y)=df$, then $\fh_X(\theta)=\fh_X(f)$ by the integrations by parts.
\end{itemize}
Let $W_{g,n}$ be the integral of $\omega_{g,n}$ over variables $X_{i}$, namely:
\begin{itemize}
    \item Define the B-model disk potential $W_{0,1}(q;X)$ as a Laurent series in $X$ with constant term zero such that
        \[
            \frac{\partial W_{0,1}(\bold{q};X)}{\partial q_{0}} = \frac{\partial\fh_X\Big(ydx\Big)}{\partial q_{0}}. 
        \]
     \item Let $\tilde{\omega}_{0,2}(Y_1,Y_2)$ be the meromorphic 2-form on $\CC^*\times \CC^*$,
    $$\tilde{\omega}_{0,2}(Y_1,Y_2):=\omega_{0,2}(Y_1,Y_2)-\frac{dX_1dX_2}{(X_1-X_2)^2}.$$
    Let $D_{\epsilon}^{*}=\{z|z\in \CC, 0<|z|<\epsilon\}$. For some small $\epsilon$,  $\tilde{\omega}_{0,2}$ is holomorphic on $(D_\epsilon^{*})^{2}$.
    Define the B-model annulus invariants by
        \[
            W_{0,2}(\bold{q};X_1,X_2) = \fh_{X_1,X_2}(\tilde{\omega}_{0,2}(Y_1,Y_2)).
        \]
    \item For $2g-2+n>0$, $\omega_{g,n}$ is holomorphic at $(D_\epsilon^{*})^{n}$, we define
        \[
            W_{g,n}(\bold{q}; X_1,\dots,X_n) = \fh_{X_1,\dots,X_n}(\omega_{g,n}).
        \]
\end{itemize}
\subsection{Hurwitz numbers}
In this section we define the Hurwitz number following the notation in \cite{BSLM14}.
Let $D(\mu)=\frac{\mu^{[\frac{\mu}{r}]}}{[\frac{\mu}{r}]!}$,
we define 
\begin{equation*}
    H_{g,\vec\mu}(\bold{t},X)=r^{n-\sum_{i=1}^{l}\delta_{\{\frac{\mu_{i}}{r}\},0}}\dprod_{i=1}^{n}D(\mu_{i})
\langle\!\langle{\frac{\bold{1}_{-\mu_1}}{1-\mu_1\psi_{1}},\dots,\frac{\bold{1}_{-\mu_n}}{1-\mu_n\psi_{n}}}\rangle\!\rangle
\end{equation*}
and 
\begin{equation*}
    H_{g,n}=\dsum_{l(\mu)=n}H_{g,\mu}X^{\mu}
\end{equation*}
We also write it as:
\begin{eqnarray*}
    H_{g,n}&= \sum_{k_1,\dots,k_n\geq 0}\sum_{\alpha_1,\dots,\alpha_n \in \{0,1,\dots,r-1\}}
        \llangle\phi_{\alpha_1}\psi^{k_1},\dots,\phi_{\alpha_n}\psi^{k_n}\rrangle_{g,n}^{[\CC/\ZZ_{r}],\TT}
        \prod_{j=1}^{n}\tilde{\xi}_{k_j}^{\alpha_j}(X_j)
        \\
                &= [z_1^{-1}\dots z_n^{-1}]\sum_{\alpha_1,\dots,\alpha_n\in\{0,1,\dots,r-1\}}
        \llangle\frac{\phi_{\alpha_1}}{z_1-\psi_1},\dots,\frac{\phi_{\alpha_n}}{z_n-\psi_n}\rrangle_{g,n}^{[\CC/\ZZ_{r}],\TT}
        \prod_{j=1}^{n}\tilde{\xi}^{\alpha_j}(z_j,X_j).
\end{eqnarray*}
where 
\begin{equation*}
    \dsum_{\alpha=0}^{r-1}\tilde{\xi}_{k}^{\alpha}(X)\phi_{\alpha}=\dsum_{\mu=0}^{\infty}r^{1-\delta_{\{\frac{\mu_{i}}{r}\},0}}D(\mu)\mu^{k}\bold{1}_{-\mu}
\end{equation*}
and
\begin{equation*}
    \tilde{\xi}^{\alpha}(z,X)=\dsum_{k=0}^{\infty}\tilde{\xi}_{k}^{\alpha}(X)z^{k}
\end{equation*}
\begin{eqnarray*}
    H_{0,1}(\tau(\bold{q}),\widetilde{X})=&\dsum_{\mu\in\ZZ_{\geq0}}r^{1-\delta_{\{\frac{\mu}{r}\},0}}D(\mu)\mu^{-2}\big(I(\tilde{\bold{q}},\frac{1}{\mu}),\bold{1}_{-\mu}\big)\widetilde{X}^{\mu}\\
    =&r^{-\delta_{\{\frac{\mu}{r}\},0}}\dsum_{(\mu,d),\mu\equiv|d|(\mod r)}\frac{\mu^{[\deg d-1]+[\frac{\mu}{r}]}}{(\frac{\mu}{r})^{\delta_{0,\{\frac{|d|}{r}\}}}(\frac{\mu-|d|}{r})!\dprod_{i=0}^{r-1}d_{i}!}\cdot \bold{q}^{d}\widetilde{X}^{\mu}\\
    =&\dsum_{(\mu,d),\mu\equiv|d|(\mod r)}\frac{\mu^{\deg d+\frac{\mu}{r}-2}}{(\frac{\mu-|d|}{r})!\dprod_{i=0}^{r-1}d_{i}!}\cdot \bold{q}^{d}\widetilde{X}^{\mu}
\end{eqnarray*}
Let $\Phi=x(\bold{q})dy$. Let 
\[
    \Phi_0(\bold{q}) := \frac{\partial \Phi}{\partial q_0} = \frac{dY}{Y}.
\]

Consider 
\[
    \mathfrak{h}_X(\Phi_0(\bold{\tilde{q}})) = \sum_{\mu\in\ZZ_{\neq 0}} R_\mu X^\mu,
\]
where
\begin{eqnarray*}
    R_{\mu}=&\frac{1}{\mu}\Res_{Y\to0}(\Phi_0(\bold{q})X^{-\mu})\\
    =&\frac{1}{\mu}\cdot e^{{\mu}(Y^{r}+\dsum_{i=0}^{r-1}q_{i}Y^{i})}[Y^{\mu}]\\
    =&\dsum_{|d|\equiv\mu(\mod r)}\dprod_{i=0}^{r-1}q_{i}^{d_{i}}\frac{\mu^{\deg d+\frac{\mu}{r}-1}}{(\frac{\mu-|d|}{r})!\dprod_{i=0}^{r-1}d_{i}!}
\end{eqnarray*}
Then we find out:
\begin{equation}
    \frac{\partial H_{0,1}(\tau(\bold{q}),\widetilde{X})}{\partial q_{0}}[\tX^{\mu}]=\mathfrak{h}_X(\Phi_0(\bold{q}))[X^{\mu}]
\end{equation}
If we identify $\widetilde{X}$ with $X$ we have the following proposition:
\begin{prop}
    $H_{0,1}(\tau(\bold{q}),X)=W_{0,1}(\bold{q},X)$
\end{prop}
\subsection{Identification of open leaves}
We substitute ordianry leaves in the $A$-model and $B$-model graph sum with open leaves. 

In $A$-model, to each ordinary leaf $l_j\in L^o(\vec{\Gamma})$ with
$\beta(l_j)=\beta\in\{0,1,\dots r-1\}$ and $k(l_j)=k\in \ZZ_{\geq 0}$, we assign the following weight (open leaf)
\begin{equation}\label{eqn:A-model-open-leaf}
    (\cL^O)^\beta_k(l_j) = [z^k](\sum_{\alpha,\gamma\in\{0,1,\dots r-1\}}
    \tilde{\xi}^\alpha(z,X_j)S^{\hat{\underline{\gamma}}}_{\ \alpha}(z)
    R(-z)_\gamma^{\ \beta}).
\end{equation}
In $B$-model, the open leaf is defined as:
        \[
            (\tilde{\cL}^O)^{\beta(l_j)}_{k(l_j)}(l_j) = \frac{1}{\sqrt{-2}} \fh_{X_j}(d\xi_{\beta(l_j),k(l_j)}(Y_j)).
        \]
We define 
\[
    U^{\hat\beta}(z)(\tau(\bold{q}),X) := \sum_{\alpha} \tilde{\xi}^\alpha(z,X)(\phi_{\hat\beta},S(\phi_\alpha))\Big.
\]
For $m\in\ZZ_{\geq -2}$, 
\[
    \left(X\frac{d}{dX}\right)[z^{m}](U^{\hat\beta}(z)(\tau(\bold{q}),X)) = [z^{m+1}](U^{\hat\beta}(z)(\tau(\bold{q}),X)).
\]
From the definition above, we have:
    \begin{equation*}
       \frac{\partial H(\tau(\bold{q}),X)}{\partial t^{i}(q)}=\dsum_{\alpha}(X^{i},S(\phi_{\alpha})(z))\tilde{\xi}^{\alpha}(z,X)[z^{-1}]
    \end{equation*}
    for $i=0,\dots,r-1$
On the B-model side, we have:
\begin{lemma}
Let $\{\bu^{\hat\alpha}(q)\}$ be the coordinates with respect to normalized canonical basis $\phi_{\hat\alpha}(q)$, we have:
    \begin{equation*}
    \frac{\partial\fh_{X}(\Phi_{0})}{\partial u^{\hat\alpha}(q)}=\fh_{X}(\frac{1}{\sqrt{-2}}\xi_{\alpha,0})
\end{equation*}
\end{lemma}
\begin{proof}
Using the fact that: 
    \begin{equation*}
        \frac{\partial \fh_{X}\big(\Phi_{0})}{\partial q_{i}}=\fh_{X}\big(\frac{dy}{dx}\cdot\frac{dx}{dq_{i}}\big)=\fh_{X}(Y^{i}(\frac{\partial W_{\TT}}{\partial y})^{-1}\big)
    \end{equation*}
We have:
\begin{equation*}    \frac{\partial\fh_{X}\big(\Phi_{0}\big)}{\partial u^{\hat\alpha}(q)}=\fh_{X}\big(\frac{\phi_{\alpha}(q)(\Delta^{\alpha}(q))^{\frac{1}{2}}}{\frac{\partial W_{\TT}}{\partial y}}\big)=\fh_{X}(\frac{1}{\sqrt{-2}}\xi_{\alpha,0})
\end{equation*}
\end{proof}
Using this lemma, we have:
for $\alpha\in\{0,1,\dots,r-1\}$, $k> 0$,
\begin{equation}\label{eqn:A-B-open}
    \begin{aligned}
        [z^k]\sum_{\beta}\tilde{\xi}^\beta(z,X)S(\hat{\phi}_\alpha(q),\phi_\beta)&= \fh_X\Big(
        (-\frac{d}{dx})^k\Big(\frac{-1}{\sqrt{-2}}\xi_{\alpha,0}\Big)\Big),
        \\
        \sum_{\beta}\tilde{\xi}^\beta(z,X)S(\hat{\phi}_\alpha(q),\phi_\beta)&= \fh_X\Big(\sum_{k\geq 0} 
        \frac{-1}{\sqrt{-2}}W^\alpha_k z^k\Big) 
        = -\fh_X\Big(\frac{\hat{\theta}_\alpha(z)}{\sqrt{-2}}\Big).
    \end{aligned}
\end{equation}
\begin{lemma}\label{lma:theta-R-matrix}
    We have
    \[
        \theta_\alpha(z) = \sum_{\beta\in\{0,1,\dots,r-1\}}\check{R}_\beta^{\ \alpha}(-z)\hat{\theta}_\beta(z).
    \]
\end{lemma}
\begin{proof}
    The lemma follows from
    \[
        d\xi_{\alpha,k} = \sum_{i=0}^{k}\sum_{\beta\in\{0,\dots,r-1\}}([z^{k-i}]\check{R}_\beta^{\ \alpha}(-z))W_i^\beta,
    \]
    which is shown in the proof of \cite[Theorem A]{FLZ17}.
\end{proof}
Then we have the identification of open leaves:
\begin{thm}\label{thm:open-leaves}
    For each ordinary leaf $l_j$ with $\beta(l_j)=\beta$ and $k(l_j)=k\in\ZZ_{\geq 0}$, we have
    \[
        (\cL^O)^\beta_k(l_j)= (\tilde{\cL}^O)^\beta_k(l_j).
    \]
\end{thm}
\begin{proof}
    By Proposition \ref{A-B R-matrix equivalence}, the A-model and B-model $R$-matrices are equal:
    $$R(z)=\tilde{R}(z).$$
    Then Theorem \ref{thm:open-leaves} follows from Equation \eqref{eqn:A-model-open-leaf}, Equation \eqref{eqn:A-B-open}, and Lemma \ref{lma:theta-R-matrix}.
\end{proof}
\subsection{Identification of $H_{0,2}$ and $W_{0,2}$}
Define 
\begin{equation*}
    \begin{aligned}
        C(Y_1,Y_2) &:= (-\frac{\partial}{\partial x(Y_1)}-\frac{\partial}{\partial x(Y_2)})
        \Big(\frac{\omega_{0,2}}{dx(Y_1)dx(Y_2)}\Big)(Y_1,Y_2)dx(Y_1)dx(Y_2)
        \\
        & \ = \Big(-d_1\circ\frac{1}{dx(Y_1)}-d_2\circ\frac{1}{dx(Y_2)}\Big)(\tilde{\omega}_{0,2}(Y_1,Y_2)).
    \end{aligned}
\end{equation*}

The following proposition is proved in \cite[Lemma 6.9]{FLZ20}.
\begin{prop} We have
\[
    C(Y_1,Y_2) = \frac{1}{2}\sum_{\alpha}d\xi_{\alpha,0}(Y_1)d\xi_{\alpha,0}(Y_2).
\]
\end{prop}
\begin{thm}
   \label{thm:annulus}
    \[
        H_{0,2}(\tau(\bold{q}),X_1,X_2)=-W_{0,2}(\bold{q},X_1,X_2).
\]
\end{thm}
\begin{proof}

\begin{equation*}
    \begin{aligned}
    & \ \ \ \ \ \mathfrak{h}_{X_1,X_2}(C)= 
    \mathfrak{h}_{X_1,X_2}(\frac{1}{2}\sum_{\alpha\in\{1,2\}}d\xi_{\alpha,0}(Y_1)d\xi_{\alpha,0}(Y_2))
    \\
    &= -[z_1^0z_2^0]\sum_{\alpha,\beta,\gamma}
    \tilde{\xi}^\beta(z_1,X_1)\tilde{\xi}^\gamma(z_2,X_2)
    S(\hat{\phi}_\alpha(q),\phi_\beta)(z_1)S(\hat{\phi}_\alpha(q),\phi_\gamma)(z_2)
    \\
    &= -(X_1\frac{\partial}{\partial X_1}+X_2\frac{\partial}{\partial X_2})H_{0,2}(\tau(\bold{q});X_1,X_2).
    \end{aligned}
\end{equation*}
By the integrations by parts,
\begin{equation*}
    \begin{aligned}
        \fh_{X_1,X_2}(C) &= \fh_{X_1,X_2}\Big(\Big(-d_1\circ\frac{1}{dx(Y_1)}-d_2\circ\frac{1}{dx(Y_2)}\Big)\tilde{\omega}_{0,2}\Big) 
        \\
        &= \Big(X_1\frac{\partial}{\partial X_1}+X_2\frac{\partial}{\partial X_2}\Big)W_{0,2}(\bold{q};X_1,X_2).
    \end{aligned}
\end{equation*}
Since the constant terms of $F_{0,2}$ and $W_{0,2}$ are both zeros, the theorem follows immediately.
\end{proof}
\subsection{Identification of Graph Sums}
We have also have identification of graph sum similar to the identification in Theorem \ref{main}, by replacing the set of the ordinary leaves $L^{o}(\vec\Gamma)$ into open leaves $L^{O}(\vec\Gamma)$. The definition of open leaves follows from the previous subsection.

For $A$-model graph sum, we assign the weight:
\begin{equation}
    w^{O}(\vec{\Gamma}) = \dprod_{v\in V (\Gamma)} \mathcal{V}^{\beta(v)}_{g(v)}(v)
\dprod_{e\in E(\Gamma)} \mathcal{E}_{k(h_1(e)),k(h_2(e))}^{\beta(v_1(e)),\beta(v_2(e))} (e) \cdot \dprod_{\ell\in L^O(\Gamma)} (\mathcal{L}^O)_{k(\ell)}^{\beta(\ell)}(\ell) \dprod_{\ell\in L^1(\Gamma)} (\mathcal{L}^1)_{k(\ell)}^{\beta(\ell)}(\ell).
\end{equation}
We have the graph sum formula:
\begin{equation}
    \dsum_{g\geq 0}\hbar^{g-1} \dsum_{n\geq 0} H_{g,n} (\bold{t},\widetilde{X}_1,\dots,\widetilde{X}_n)= \dsum_{g\geq 0}\hbar^{g-1} \dsum_{n\geq 0} \dsum_{\vec{\Gamma}\in \mathbf{\Gamma}_{g,n}(\cF(m,r))} \dfrac{w^{O}(\vec{\Gamma})}{|\mathrm{Aut}(\vec{\Gamma})|}.
\end{equation}
For the $B$-model graph sum, we assign the weight:
\begin{equation}
    \tilde{w}^{O}(\vec{\Gamma}) = (-1)^{g(\vec{\Gamma})-1+N} \dprod_{v\in V (\Gamma)} \tilde{\mathcal{V}}^{\alpha(v)}_{g(v)}(v)
\dprod_{e\in E(\Gamma)} \check{B}_{k(e),\ell(e)}^{\alpha(v_1(e)),\alpha(v_2(e))} (e) \dprod_{\ell\in L^O(\Gamma)} (\tilde{\mathcal{L}}^O)_{k(\ell)}^{\alpha(\ell)}(\ell)\dprod_{\ell\in L^1(\Gamma)} \big( \dfrac1{\sqrt{-2}} \big)\check{h}_{k(\ell)}^{\alpha(\ell)}.
\end{equation}
We have the graph sum fomula:
\begin{equation}    
W_{g,n}(\bold{q},X_{1},\dots,X_{n})=\dsum_{\vec{\Gamma}\in \mathbf{\Gamma}_{g,n}(\cF(m,r))} \dfrac{\tilde{w}^{O}(\vec{\Gamma})}{|\mathrm{Aut}(\vec{\Gamma})|}
\end{equation}
for $2g-2+n>0$, and $n>0$.

Combining this result with the identification $(0,1)$, $(0,2)$ case, we have the following theorem.
\begin{thm}
By identifying $\widetilde{X}_{i}$ with $X_{i}$, we have
   \[
W_{g,n}(\bold{q},X_{1},\dots,X_{n})=(-1)^{g-1+N}H_{g,n} (\tau(\bold{q}),\widetilde{X}_1,\dots,\widetilde{X}_n)
\] 
for $n>0$.
\end{thm}

\subsection{Bouchard-Mari\~{n}o conjecture on q$\to$0 limit}
In this subsection, we will further specialize Theorem \ref{main} to the $q\to0$ case. In this case, Theorem \ref{main} relates the invariant $\omega_{g,n}$ of the limit curve to the equivariant descendent theory of $\CC$. After expanding $\xi_{\alpha,0}$ in suitable coordinates, we can relate the corresponding expansion of $\omega_{g,n}$ to the generation function of Hurwitz numbers and therefore reprove the Bouchard-Mari\~{n}o conjecture \cite{BM08} on Hurwitz numbers.

Let $\w_i=0$ for $i\neq r$, $q_i=0$ for all $i\in I_{0}$, and take the large radius limit $q_{*}\to 0$. Then our mirror curve becomes
$$x=Y^{r}-r\w_{r}\log Y.$$
When $\w_{r}=\frac{1}{r}$, this is just the $r$-Lambert curve:
\begin{equation*}
    X=Ye^{-Y^{r}}
\end{equation*}
when setting $X=e^{-x}$.

We notice that in this limit we have $n$ critical points $P_{\alpha}$ of $W_{\TT}(Y)$ which are:
$$P_{\alpha}=\zeta_{r}^{\alpha}(\bold{w}_{r})^{\frac{1}{r}},$$
where $\zeta_{r}=\exp(\frac{2\pi i}{r})$, and $\alpha=0,1,\dots,r-1$.

We note that under the large radius limit, other $m$ critical points turns to zero, which is out of the curve.

Under the identification $\frac{1}{\sqrt{-2}}W^\alpha_k(Y_j)=(\tilde{u}_j)^{\hat\alpha}_k$ in Theorem \ref{main}, when $q_{*}\to 0$, $m$ of the critical points goes to 0. On the A-model side, since $\bold{q}=0$, the $S-$matrix $(S^{\hat{\alpha}}_{\hat{\beta}}(z))$ is diagonal. Therefore, $m$ of $(u_j)^\alpha_k\to 0$ when $q\to 0$ under the identification in Theorem \ref{main}. This means that in the localization graph of the equivariant GW invariants of $\cF(m,r)$, we can only have a constant map to $p_2\in\cF(m,r)$. Since $H|_{p_2}=m\w_m=0$ and $q_{i}=0$, we can not have any primary insertions.
Therefore, in the large radius limit, we get
\begin{eqnarray*}
F^{\cF(m,r),\CC^*}_{g,n}(\bold{u}_1,\cdots,\bold{u}_{n};\bt) &=& \sum_{a_1,\ldots, a_n\in \ZZ_{\geq 0}}
\int_{[\Mbar_{g,n}(\cF(m,r),0)]^\vir}
\prod_{j=1}^{n} \ev_j^*(\sum_{\alpha}(u_j)^{\alpha}_{a_j}\phi_\alpha(0)) \psi_j^{a_j}\\
&=&\sum_{a_1,\ldots, a_n\in \ZZ_{\geq 0}}\sum_{\alpha}
\int_{\Mbar_{g,n}(\mathcal{B}\ZZ_{r})}\prod_{j=1}^{n} (u_j)^\alpha_{a_j} \psi_j^{a_j}\phi_{\alpha}e_{\CC^{*}}^{-1}(\EE),
\end{eqnarray*}
where $\EE$ is the hodge bundle, 
and we also write $\lambda_j=c_j(\bE)$ is the $j$-th Chern class of the Hodge bundle. Then we have
\begin{equation*}
    e_{\CC^{*}}(\EE)=\w_{r}^{1-g-\sum_{\gamma_{i}}\age(\gamma_{i})}\bold{c}_{\frac{1}{\w_{r}}}(\EE)
\end{equation*}
on $\Mbar_{g,\vec\gamma}(\mathcal{B}\ZZ_{r})$.

At the same time, we also have $\Delta^{\alpha}(0)=r^{2}\w_{r}$. So $\frac{(u_j)^\alpha_k}{r\sqrt{\w_{r}}}=(\tilde{u}_j)^{\hat\alpha}_k$. Therefore Theorem \ref{main} specializes to
$$\omega_{g,n}|_{\frac{1}{\sqrt{-2}}W^\alpha_k(Y_j)=
\frac{(u_j)^\alpha_k}{r\sqrt{\w_{r}}}}=(-1)^{g-1+r}\sum_{a_1,\ldots, a_n\in \ZZ_{\geq 0}}\sum_{\alpha}
\frac{1}{r\w_{r}}\int_{\Mbar_{g,n}(\mathcal{B}\ZZ_{r})}\prod_{j=1}^n (u_j)^\alpha_{a_j} \psi_j^{a_j}\Lambda_{g}^{\vee}(-\w_{r}).$$

Now we study the expansion of $\xi_{\alpha,0}$ near the point $Y=0$ in the coordinate $Z=e^{-\frac{x}{r\w_r}}$.
We have
$$
\xi_{\alpha,0}=\frac{1}{\sqrt{-1}}\sqrt{\frac{2}{r^2\w_{r}}}\frac{\zeta_{r}^{\alpha}(\w_{r})^{\frac{1}{r}}}{Y-\zeta_{r}^{\alpha}(\w_{r})^{\frac{1}{r}}}.
$$
Since $Z=Ye^{-\frac{Y^{r}}{r\w_{r}}}$, by taking the differential we have
$$\frac{dZ}{Z}=-\frac{Y^{r}-\w_{r}}{Y\w_{r}}dY.$$
Therefore, $\frac{dZ}{Z}\xi_{\alpha,0}=-\frac{1}{\sqrt{-1}}\sqrt{\frac{2}{r^2\w_{r}}}\sum_{i=0}^{r-1}(\frac{Y}{P_{\alpha}})^{i}\frac{dY}{Y}$.
Let
$$\xi_{\alpha,0}=\sum_{\mu=0}^{\infty}C_\mu Z^\mu.$$
near the point $Y=0$. Then we have
\begin{align*}
C_\mu &=\Res_{Y\to 0}\xi_{\alpha,0}Z^{-\mu}\frac{dZ}{Z}\\
&=-\frac{1}{\sqrt{-1}}\sqrt{\frac{2}{r^2\w_{r}}}\Res_{Y\to 0} (e^{\frac{\mu Y^k}{r\w_{r}}}\sum_{i=0}^{r-1}(\frac{Y}{P_{\alpha}})^{i}\frac{dY}{Y^{\mu+1}})\\
&=-\frac{1}{\sqrt{-1}}\sqrt{\frac{2}{r^2\w_{r}}}\frac{(\frac{\mu}{r\w_{r}})^{[\frac{\mu}{r}]}}{[\frac{\mu}{r}]!}P_{\alpha}^{-r\{\frac{\mu}{r}\}}.
\end{align*}
Therefore
$$
W^\alpha_k=-\frac{1}{\sqrt{-1}}\sqrt{\frac{2}{r^2\w_{r}}}r\w_{r}\sum_{\mu=0}^{\infty}\frac{(\frac{\mu}{r\w_{r}})^{[\frac{\mu}{r}]}}{[\frac{\mu}{r}]!}(\frac{\mu}{r\w_{r}})^{k+1}P_{\alpha}^{-r\{\frac{\mu}{r}\}}Z^{\mu-1}dZ.
$$

On A-model side, let
$$
(u_j)^\alpha_{a_j}=\sum_{\mu_j=0}^{\infty}\frac{(\frac{\mu_j}{r\w_{r}})^{[\frac{\mu_j}{r}]}}{[\frac{\mu_j}{r}]!}
(\frac{\mu_j}{r\w_{r}})^{a_j}P_{\alpha}^{-r\{\frac{\mu}{r}\}}Z^{\mu_j}_j.
$$
Then if we let $C(\mu_{j})=\frac{(\frac{\mu_j}{r\w_{r}})^{[\frac{\mu_j}{r}]}}{[\frac{\mu_j}{r}]!}$, we have:
\begin{equation*}
    \sum_{\alpha}(u_j)^\alpha_{a_j}\phi_{\alpha}(0)=C(\mu_{j})(\frac{\mu_j}{r\w_{r}})^{a_j}X^{r-r\{\frac{\mu_{j}}{r}\}}\w_{r}^{-1}
\end{equation*}
Then we have
\begin{eqnarray*}
&& F^{[\CC/\ZZ_{r}],\CC^*}_{g,n}(\bold{u}_1,\cdots,\bold{u}_n)[Z^{\vec{\mu}}]\\
&=&  
\w_{r}^{-n+\sum_{i=1}^{n}\delta_{\{\frac{\mu_{i}}{r}\},0}}\int_{\Mbar_{g,n}(\mathcal{B}\ZZ_{r})}e_{\CC^{*}}^{-1}(\EE)
\prod_{j=1}^{n}C(\mu_{j})
\frac{1}{1-\frac{\mu_{j}\psi_{j}}{r\w_{r}}}\ev^{*}_{j}(1_{r\{-\frac{\mu_{j}}{r}\}})\\
&=& \w_{r}^{-n+\sum_{i=1}^{n}\delta_{\{\frac{\mu_{i}}{r}\},0}}\int_{\Mbar_{g,-\vec{\mu}}(\mathcal{B}\ZZ_{r})}e_{\CC^{*}}^{-1}(\EE)
\prod_{j=1}^{l}C(\mu_{j})
\frac{1}{1-\frac{\mu_{j}\psi_{j}}{r\w_{r}}}.
\end{eqnarray*}
By the r-ELSV formula \cite{JPT11}
\begin{eqnarray*}
H_{g,\vec\mu}^{(r)} &=& 
r^{1-g+\sum_{i=1}^{n}\{\frac{\mu_{i}}{r}\}}\prod_{j=1}^n \frac{[\frac{\mu_j}{r}]^{\mu_j}}{[\frac{\mu_j}{r}]!}
\int_{\Mbar_{g,-\vec{\mu}}(\mathcal{B}\ZZ_{r})}\frac{\sum_{i=0}^{\infty}(-r)^{i}\lambda_{i}}{\prod_{j=1}^{n}(1-\mu_j\psi_{j})}\\
&=&r^{n-\sum_{i=1}^{l}\delta_{\{\frac{\mu_{i}}{r}\},0}}\prod_{j=1}^n \frac{[\frac{\mu_j}{r}]^{\mu_j\psi_{j}}}{[\frac{\mu_j}{r}]!}
\int_{\Mbar_{g,-\vec{\mu}}(\mathcal{B}\ZZ_{r})}\frac{e_{\CC^{*}}^{-1}(\EE)(r\w_{r})^{3g-3+n+\rank \EE}}{\prod_{j=1}^n(1-\frac{\mu_j\psi_{j}}{r\w_{r}})}.\\
\end{eqnarray*}

So we have: %\xxx{sign problems}
$$
F^{[\CC/\ZZ_{r}],\CC^*}_{g,n} = \sum_{\vec\mu\in(\ZZ_{\geq0})^{n}}
H_{g,\vec\mu}^{(r)} Z^{\vec\mu}.
$$
when $\w_{r}=\frac{1}{r}$, which is just the generating function of Hurwitz numbers.

Let $W_{g,n}(Z_1,\cdots,Z_l)=\int^{Z_1}_0\cdots\int^{Z_n}_0\omega_{g,n}(Y_1(Z_{1}),\cdots,Y_l(Z_{l}))$, where the integral is over $Z_{i}$ variables.  Then we have

\begin{cor}[Orbifold Bouchard-Mari\~{n}o conjecture]\cite{BSLM14}
For $n>0$ and $2g-2+n>0$, the invariant $W_{g,n}(Z_1,\cdots,Z_n)$ for the curve $x=Y^{r}-r\w_{r}\log Y$ satisfies
\begin{eqnarray*}
&& W_{g,n}(Z_1,\cdots,Z_n)\\
&=& (-1)^{g-1+n} \w_{r}^{n-\sum_{i=1}^{l}\delta_{\{\frac{\mu_{i}}{r}\},0}}\int_{\Mbar_{g,-\vec{\mu}}(\mathcal{B}\ZZ_{r})}e_{\CC^{*}}^{-1}(\EE)
\prod_{j=1}^{l}C(\mu_{j})
\frac{1}{1-\frac{\mu_{j}\psi_{j}}{r\w_{r}}}\\
&=& (-1)^{g-1+n}\sum_{\vec\mu\in(\ZZ_{\geq0})^{l}}
(r\w_{r})^{\sum_{i=1}^{l}\{\frac{\mu_{i}}{r}\}+\{-\frac{\mu_{i}}{r}\}}H_{g,\mu}^{(r)} Z^{\vec\mu}.
\end{eqnarray*}
In particular, when $\w_{r}=\frac{1}{r}$, the right hand side is the generating function of Hurwitz numbers and the Bouchard-Mari\~{n}o conjecture is recovered.
\end{cor}
\section{The non-equivariant limit and the Norbury-Scott conjecture}

In this section, we consider the non-equivariant limit $\w_i=0$, and $q_{i}=0$ for $I\in I_{0}$. Under this limit, we get a generalized version of Norbury-Scott conjecture \cite{NS14} for $\cF(m,r)$.
\subsection{Differentials of $J$-function under $q_{i}\to0$ limit}
We first calculate the first order differentials of non-equivariant $J$-function under $q_{i}\to0$ limit. The only non-classical quantum product is $\bold{1}_{1}*\bold{1}_{-1}=q_{*}$.

In $q_{i}\to0$ limit, the mirror map $\tau(q)$ reduces to $t_{i}=q_{i}$ for $i\in I_{0}$, and $t=\log q_{*}$. Let $\tilde{t}_{i}$ be the coordinate of the flat basis $X^{i}$ for $i=0,1,\dots,m+r-1$. In this case, we have $\tilde{t}_{i}=q_{i}$ for $0\leq i< n$, and $\tilde{t}_{i}=q_{*}^{m+r-i}q_{i-m-r}$ for $n<i<m+r$, and $\tilde{t}_{r}=r\log q_{*}$

Using Non-equivariant Mirror Theorem (\ref{non-equivariant mirror theorem}), we first notice: 
\begin{equation*}
    \big(\frac{z\partial J(z)}{\partial{\tilde{t}}_{0}}\big)_{q_{i}\to0}=e^{\frac{H\log q_{*}}{z}}\dsum_{m|l \ or\ r|l,l\geq0}\frac{\Gamma(\frac{H}{mz}+1-{\{-\frac{l}{m}}\})}{\Gamma(\frac{H}{mz}+\frac{l}{m}+1)}\cdot\frac{\Gamma(\frac{H}{rz}+1-   {\{-\frac{l}{r}}\})}{\Gamma(\frac{H}{rz}+\frac{l}{r}+1)}q_{*}^{l}z^{-\lceil
    \frac{l}{m}+\frac{l}{r}\rceil}\bold{1}_{v(lp_{*})},
\end{equation*}
\begin{eqnarray*}
    \big(\frac{z\partial J(z)}{\partial{\tilde{t}}_{r}}\big)_{q_{i}\to0}&= e^{\frac{H\log q_{*}}{z}}\dsum_{m|l \ or\ r|l,l\geq0}\frac{\Gamma(\frac{H}{mz}+1-{\{-\frac{l}{m}}\})}{\Gamma(\frac{H}{mz}+\frac{l}{m}+1)}\cdot\frac{\Gamma(\frac{H}{rz}+1-   {\{-\frac{l-i}{r}}\})}{r\Gamma(\frac{H}{rz}+\frac{l}{r}+1)}lq_{*}^{l}z^{-\lceil\frac{l}{m}+\frac{l}{r}\rceil}\bold{1}_{v(lp_{*})}\\
    &+\dsum_{m|l,\ r|l,l\geq0}\frac{q_{*}^{l}}{(\frac{l}{m})!(\frac{l}{r})!k}z^{-\frac{l}{m}-\frac{l}{r}-1}H
\end{eqnarray*}
and for $0<i<r$
\begin{equation*}
    \big(\frac{z\partial J(z)}{\partial{\tilde{t}}_{i}}\big)_{q_{i}\to0}=    e^{\frac{H\log q_{*}}{z}}\dsum_{m|l \ or\ r|l-i,l\geq 0}\frac{\Gamma(\frac{H}{mz}+1-{\{-\frac{l}{m}}\})}{\Gamma(\frac{H}{mz}+\frac{l}{m}+1)}\cdot\frac{\Gamma(\frac{H}{rz}+1-   {\{-\frac{l-i}{r}}\})}{\Gamma(\frac{H}{rz}+\frac{l-i}{r}+1)}q_{*}^{l}z^{-\lceil\frac{l}{m}+\frac{l-i}{r}\rceil}\bold{1}_{v(lp_{*}+p_{i})},
\end{equation*}
and for $r<i<m+r$
\begin{equation*}
    \big(\frac{z\partial J(z)}{\partial{\tilde{t}}_{i}}\big)_{q_{i}\to0}=    e^{\frac{H\log q_{*}}{z}}\dsum_{m|l \ or\ r|l-i,l\geq i-r}\frac{\Gamma(\frac{H}{mz}+1-{\{-\frac{l}{m}}\})}{\Gamma(\frac{H}{mz}+\frac{l-m}{m}+1)}\cdot\frac{\Gamma(\frac{H}{rz}+1-   {\{-\frac{l-i}{r}}\})}{\Gamma(\frac{H}{rz}+\frac{l+r-i}{r}+1)}q_{*}^{l-m}z^{-\lceil\frac{l}{m}+\frac{l-i}{r}\rceil}\bold{1}_{v(lp_{*}+p_{i-m-r})},
\end{equation*}
Given $m,r$, we define $D=lcm(m,r)$, and $d=gcd(m,r)$.

Under the flat basis $\{X^{i}\}$ for $i=0,1\dots,m+r-1$. In $q_{i}\to0$ limit, we have the following correspondence:
\begin{align*}
    X^{i}=&1_{i} \ \mathrm{for}\ 0\leq i<r \\
    X^{r}=&\frac{H}{r}  \\
    X^{i}=&q_{*}^{i-r}1_{i-m-r} \ \mathrm{for}\ r< i<m+r \\
\end{align*}
We write $\tilde{\partial}_{i}J:=\big(\frac{z\partial J(z)}{\partial{\tilde{t}}_{i}}\big)$ for short. Rewriting the equations as a power series of $\{X^{i}\}$, we have:
\begin{equation*}    \tilde{\partial}_{0}J(z)_{q_{i}\to0}=\dsum_{m|l\ or \ r|l,l\geq0}\frac{\Gamma(1-{\{-\frac{l}{m}}\})}{\Gamma(\frac{l}{m}+1)}\cdot\frac{\Gamma(1-{\{-\frac{l}{r}}\})}{\Gamma(\frac{l}{r}+1)}(\frac{r}{m})^{\lceil\frac{l}{m}\rceil}(\frac{X^{r}}{z})^{\lceil
    \frac{l}{m}+\frac{l}{r}\rceil}+\dsum_{D|l,l\geq0}\frac{(\frac{r}{m})^{\frac{l}{m}}r(\frac{X^{r}}{z})^{
    \frac{l}{m}+\frac{l}{r}+1}}{(\frac{l}{m})!(\frac{l}{r})!},
\end{equation*}
\begin{equation*}    \tilde{\partial}_{r}J(z)_{q_{i}\to0}=X^{r}(\dsum_{m|l\ or \ r|l,l\geq0}\frac{\Gamma(1-{\{-\frac{l}{m}}\})}{\Gamma(\frac{l}{m}+1)}\cdot\frac{\Gamma(1-{\{-\frac{l-r}{r}}\})}{\Gamma(\frac{l-r}{r}+1)}(\frac{r}{m})^{\lceil\frac{l}{m}\rceil}(\frac{X^{r}}{z})^{\lceil
    \frac{l}{m}+\frac{l}{r}\rceil}+\dsum_{D|l,l\geq0}\frac{(r+\frac{r}{l})(\frac{r}{m})^{\frac{l}{m}}(\frac{X^{r}}{z})^{
    \frac{l}{m}+\frac{l}{r}+1}}{(\frac{l}{m})!(\frac{l-r}{r})!}),
\end{equation*}
and for $0<i<r$
\begin{equation*}
\tilde{\partial}_{i}J(z)_{q_{i}\to0}=X^{i}(\dsum_{m|l \ or\ r|l-i,l\geq 0}\frac{\Gamma(1-{\{-\frac{l}{m}}\})}{\Gamma(\frac{l}{m}+1)}\cdot\frac{\Gamma(1-{\{-\frac{l-i}{r}}\})}{\Gamma(\frac{l-i}{r}+1)}(\frac{r}{m})^{\lceil\frac{l}{m}\rceil}(\frac{X^{r}}{z})^{\lceil\frac{l}{m}+\frac{l-i}{r}\rceil}+\dsum_{m|l,r|l-i,l\geq0}\frac{(\frac{r}{m})^{\frac{l}{m}}r(\frac{X^{r}}{z})^{    \frac{l}{m}+\frac{l-i}{r}+1}}{(\frac{l}{m})!(\frac{l}{r})!}),
\end{equation*}
and for $r<i<m+r$
\begin{equation*}
\tilde{\partial}_{i}J(z)_{q_{i}\to0}=X^{i}(\dsum_{m|l \ or\ r|l-i,l\geq i-r}\frac{\Gamma(1-{\{-\frac{l}{m}}\})}{\Gamma(\frac{l}{m}+1)}\cdot\frac{\Gamma(1-{\{-\frac{l-i}{r}}\})}{\Gamma(\frac{l-i}{r}+1)}(\frac{r}{m})^{\lceil\frac{l}{m}\rceil}(\frac{X^{r}}{z})^{\lceil\frac{l}{m}+\frac{l-i}{r}\rceil}+\dsum_{m|l,r|l-i,l\geq0}\frac{(\frac{r}{m})^{\frac{l}{m}}r(\frac{X^{r}}{z})^{    \frac{l}{m}+\frac{l-i}{r}+1}}{(\frac{l}{m})!(\frac{l}{r})!}),
\end{equation*}
From the calculation above, we find out that $\tilde{\partial}_{i}J$ takes the same form except for $i=r$. And $\tilde{\partial}_{r}J$ varies from the form because $\tilde{t}_{r}$ also appears in the $e^{\frac{\tilde{t}_{r}H}{rz}}$.

\subsection{The Norbury-Scott Conjecture} \label{sec:NS14}
Consider the expansion
of $\xi_{\alpha,0}$ near $Y=0$ in $\tx=x^{-1}= (Y^{r}+q_{*}^{m}Y^{-m})^{-1}$.

% Since $$\xi_{\alpha,0}=\sqrt{\frac{2}{\Delta^
% \alpha(q)}}\frac{P_\alpha}{Y-P_\alpha},$$
% its $n$-th coefficient in the expansion of $z=(Y+\frac{q}{Y})^{-1}$ at $Y=0$ is given by the residue
% \begin{align*}
% &\mathrm{Res}_{Y=0} z^{-k-1} \xi_{\alpha,0} dz\\
% =&-\mathrm{Res}_{Y=0} (Y+\frac{q}{Y})^{k-1} (1-\frac{q}{Y^2}) \xi_{\alpha,0} dY\\
% =&\sqrt{\frac{2}{\Delta_\alpha(q)}}\left( \sum_{r\in[\frac{n}{2},r-1]\cap \bZ}\frac{q^k}{P_\alpha^{2k-n}}\left(\binom{k-1}{r}-\binom{k-1}{k-1}\right)-( \frac{q}{P_\alpha})^r\right).
% \end{align*}
% Similarly, the $n$-th coefficient in the expansion at $Y=\infty$ is
% $$
% -\sqrt{\frac{2}{\Delta_\alpha(q)}}\left( \sum_{r\in[\frac{n}{2},r-1]\cap \bZ} P_\alpha^{2k-n}q^{k-k}\left(\binom{k-1}{r}-\binom{k-1}{k-1}\right)-P_\alpha^r\right).
% $$

Assume that $q_{*} \in (0,\infty)$. We have
\begin{equation*}
P_{\alpha}=(\frac{mq_{*}^{m}}{r})^{\frac{1}{m+r}}\zeta_{m+r}^{\alpha}
\end{equation*}
\begin{equation*}
    \Delta^{\alpha}=((m+r)^{m+r}m^{r}r^{m}q_{*}^{mr})^{\frac{1}{m+r}}\zeta_{m+r}^{-m\alpha}
\end{equation*}

The $l$-th coefficient in the expansion of $\tx =(Y^{r}+q_{*}^mY^{-m})^{-1}$ at $Y=0$ is given by the residue
\begin{align*}
\mathrm{Res}_{Y=0} \tx^{-l-1} \xi_{\alpha,0} d\tx
=&\sqrt{\frac{-2}{\Delta^{\alpha}(0)}}\mathrm{Res}_{Y=0} (Y^{r}+q_{*}^{m}Y^{-m})^{l-1} \sum_{i=-m}^{r-1}r(\frac{Y}{P_{\alpha}})^{i}P_{\alpha}^{r}\frac{dY}{Y} \\
=& \sqrt{\frac{-2}{\Delta^{\alpha}(0)}}\mathrm{Res}_{Y=0} (Y^{r+m}+q_{*}^{m})^{l-1} \sum_{i=0}^{m+r-1}r(\frac{Y}{P_{\alpha}})^{i}P_{\alpha}^{m+r}[Y^{ml}] \\
=& \sqrt{\frac{-2}{\Delta^{\alpha}(0)}}r^{\{\frac{lm}{m+r}\}}m^{1-\{\frac{lm}{m+r}\}}q_{*}^{\frac{lmr}{m+r}}\zeta_{m+r}^{rl\alpha}{l-1 \choose \lfloor \frac{ml}{m+r}\rfloor}
\end{align*}

$$
\xi_{\alpha,0} =  \sqrt{\frac{-2}{\Delta^{\alpha}(0)}}\sum_{l=0}^\infty r^{\{\frac{lm}{m+r}\}}m^{1-\{\frac{lm}{m+r}\}}q_{*}^{\frac{lmr}{m+r}}{l-1 \choose \lfloor \frac{ml}{m+r}\rfloor}\zeta_{m+r}^{rl\alpha}x^{-l}
$$

This shows:
\begin{equation*}
    \tilde{u}^{\hat\alpha}_{k}=\sqrt{\frac{1}{\Delta^{\alpha}(0)}}\sum_{l=0}^\infty r^{\{\frac{lm}{m+r}\}}m^{1-\{\frac{lm}{m+r}\}}q_{*}^{\frac{lmr}{m+r}}\frac{(l+k)!}{\lfloor \frac{ml}{m+r}\rfloor!(l-1-\lfloor (\frac{ml}{m+r}\rfloor)!}\zeta_{m+r}^{rl\alpha}x^{-l-k-1}dx
\end{equation*}
\begin{equation*}    \tilde{\bu}^{\hat\alpha}=\sqrt{\frac{1}{\Delta^{\alpha}(0)}}\sum_{l,k=0}^\infty r^{\{\frac{lm}{m+r}\}}m^{1-\{\frac{lm}{m+r}\}}q_{*}^{\frac{lmr}{m+r}}\frac{(l+k)!}{\lfloor (\frac{ml}{m+r}\rfloor)!(l-1-\lfloor \frac{ml}{m+r}\rfloor!}\zeta_{m+r}^{rl\alpha}(\frac{z}{x})^{k} x^{-l-1}dx
\end{equation*}
\begin{equation*}
    \tilde{\bu}^{\tilde{i}}=\dsum_{rl-i|m+r,k=0}^\infty rq_{*}^{m[\frac{lr}{m+r}]}\frac{(l+k)!}{\lfloor (\frac{ml}{m+r}\rfloor)!(l-1-\lfloor (\frac{ml}{m+r}\rfloor)!}(\frac{z}{x})^{k}x^{-l-1}dx
\end{equation*}
This implies 
\begin{equation*}
\dsum_{\alpha}\bu^{\alpha}\phi_{\alpha}=(\phi_{\alpha})\bold{\Delta}(\bold{S}^{\ualpha}_{\ubeta})^{T}\bold{\Delta}^{\frac{1}{2}}\tilde{\bu}^{\hat\beta}=\dsum_{\beta}-z\frac{\partial J(-z)}{\partial u^{\beta}}\tilde{\bu}^{\beta}=\dsum_{i=0}^{m+r-1}\tilde{\partial}_{i}J(-z)\tilde{\bu}^{\tilde{i}}
\end{equation*}
where $\Delta=\mathrm{diag}(\Delta^{\alpha}(0))$.
Using Theorem \ref{main}, taking $z=1$ we have:
\begin{equation*} 
    \omega_{g,n}=(-1)^{g-1+n}\llangle \dsum_{i=0}^{m+r-1}\tilde{\partial}_{i}J(-1)\tilde{\bu}^{\tilde{i}}(1,x_{1}),\cdots ,\dsum_{i=0}^{m+r-1}\tilde{\partial}_{i}J(-1)\tilde{\bu}^{\tilde{i}}(1,x_{n})\rrangle_{g,n}^{\cF(m,r)}
\end{equation*}
\begin{remark}
Notice that:
    \begin{equation*}    \tilde{\bu}^{\tilde{i}}X^{i}=r\dsum_{rl-i|m+r,k=0}^\infty \frac{(l+k)!(\frac{r}{m})^{\lfloor \frac{ml}{m+r}\rfloor}}{(\lfloor \frac{ml}{m+r}\rfloor)!(l-1-\lfloor (\frac{ml}{m+r}\rfloor)!}(\frac{z}{x})^{k}X^{lr}x^{-l-1}dx
\end{equation*}
when $2g-2+n>0$ and $n>0$.
This means $\dsum_{i=0}^{m+r-1}\tilde{\partial}_{i}J(-1)\tilde{\bu}^{\tilde{i}}(1,x_{n})$ is actually a power series of $X^{r}=\frac{H}{r}$, which is also in accordance with Norbury-Scott conjecture when $m=r=1$.
\end{remark}

% Similarly, the $n$-th coefficient in the expansion at $Y=\infty$ is
% $$
% (\sqrt{q})^{k-\frac{1}{2}}{r-1 \choose \lfloor \frac{n}{2}\rfloor},\quad
% (-\sqrt{q})^{k-\frac{1}{2}}{r-1\choose \lfloor\frac{n}{2} \rfloor}.
% $$

\color{black}

\end{document}